\documentclass[11pt]{article}
\usepackage{amsmath,amsthm,amssymb,amsfonts,amsbsy,mathtools,enumerate}
\usepackage{tikz}
\usepackage{float}
\usepackage{graphicx}
\usepackage{blkarray}
\usepackage{caption, subcaption}
\usepackage{nicematrix}
\usepackage[toc,page]{appendix}
\usepackage[hidelinks]{hyperref}
\usepackage{url}
\usepackage{tikz-cd}
\usepackage{mathtools}
\usepackage[T1]{fontenc}
\usepackage[utf8]{inputenc}

\newtheorem*{theorem*}{Theorem}
\theoremstyle{definition}
\newtheorem{theorem}{Theorem}[section]
\newtheorem{corollary}[theorem]{Corollary}

\newtheorem{lemma}[theorem]{Lemma}
\newtheorem{prop}[theorem]{Proposition}

\newtheorem{defn}[theorem]{Definition}
\newtheorem{rk}{Remark}

\newcommand\sqr[2]{{\vbox{\hrule height.#2pt
    \hbox{\vrule width.#2pt height#1pt \kern#1pt
        \vrule width.#2pt}\hrule height.#2pt}}}
\newcommand\qedbox{%
	\ifmmode\eqno\sqr53
	\else\nolinebreak\ \hfill\sqr53\medbreak\fi}

\newcommand{\ld}{\ensuremath{\lambda}}
\newcommand{\Ga}{\ensuremath{\Gamma}}

\newcommand\la{\lambda}
\newcommand\cl[1]{\text{cl}(#1)}
\newcommand\lk[1]{\text{Lk}(#1)}
\newcommand\ul[1]{L_{#1}^{\uparrow}}
\newcommand\dl[1]{L_{#1}^{\downarrow}}

\def\N{{\mathbb N}}

\def\T{{\mathcal{T}}}
\newcommand\re{{\mathbb R}}%reals

\newcommand\rankg[2]{\text{rank}_{\mathbb{#1}}(#2)}

\DeclareMathOperator\im{im}

\usepackage[backend=biber,style=alphabetic,maxbibnames=99,maxalphanames=4]{biblatex}
\title{Triangle Families with Large Edge Up-Laplacian Spectral Gap} 
\author{Mutasim Mim\footnote{Department of Mathematics, CUNY Graduate Center, CUNY, NY, USA. \texttt{Email: mutasim.mim71@gc.cuny.edu}}}
\date{}
\begin{document}
\maketitle

\begin{abstract}
    Let $\T$ be a finite nonempty family of $3$-element sets, viewed as the triangles of a $2$-dimensional simplicial complex, and let $\lambda(\mathcal{T})$ be the smallest positive eigenvalue of its edge up-Laplacian $\delta_{1,\mathcal{T}}^T\delta_{1,\mathcal{T}}$, where $\delta_{1,\mathcal{T}}$ is the signed edge-triangle incidence matrix. We view $\lambda(\mathcal{T})$ as an analogue of the algebraic connectivity one dimension up, and study the extremal function $\phi(t)=\max_{|\mathcal{T}|=t}\lambda(\mathcal{T})$. For every $n \geq 3$, we show that $\lambda(\mathcal{T}) \leq n-1$ whenever $|\mathcal{T}|<\binom{n}{3}$ and that $\phi\left(\binom{n}{3}\right)=n$, attained only by the complete triple system on $n$ points. For every $n\geq 5$, we show that $\phi(t) \leq n-1$ for all $\binom{n}{3} < t < \binom{n+1}{3}$. In particular, $\phi$ is not monotone, while $\max_{s \leq t} \phi(s)=\max\left\{n : \binom{n}{3} \leq t \right\}$. We also prove that $\phi(t)=\Theta(t^{\frac{1}{3}}).$

    The proofs rest on two local bounds: $\lambda(\mathcal{T}) \leq d+2$ for every edge lying in exactly $d$ triangles of $\mathcal{T}$, and, in the spirit of Garland's local-to-global method, $\lambda(\mathcal{T}) \leq 1 + \mu$ for every nonzero Laplacian eigenvalue $\mu$ of every connected component of every vertex link. Applying Fiedler's inequality between algebraic connectivity and minimum degree to the link, the second bound gives rigidity: an edge lying in exactly $d$ triangles with $\lambda(\mathcal{T}) > d+1$ must lie in a complete triple system on $d+2$ points that shares no edge with the rest of $\mathcal{T}$. When $\mathcal{T}$ is the set of all triangles of a connected graph, the smallest positive eigenvalue of the Helmholtzian (edge Hodge Laplacian) is the minimum of $\lambda(\mathcal{T})$ and the algebraic connectivity, so both bounds apply to it. 
\end{abstract}



\section{Introduction}
Let $G$ be a graph, $L_0$ the usual graph Laplacian of $G$, and $\mu_2$ the second smallest eigenvalue of $L_0$, the algebraic connectivity of $G$. Denote by $\kappa(G)$ the vertex connectivity of $G$ and by $d_{\min}(G)$  the minimum degree of $G$. It is well-known that (see Section 1.7 in \cite{brouwer-haemers}) for a connected, non-complete graph $G$,
\begin{gather} \label{eq:algebraic-connectivity-min-degree-bound}
    \mu_2 \leq \kappa(G) \leq d_{\min}(G).
\end{gather}
Thus, a large algebraic connectivity forces local structure at every vertex. An important family of problems in extremal spectral graph theory is maximizing the algebraic connectivity given the number of vertices and edges in a graph. In this paper, we study the analogous parameter and the analogous extremal problem one dimension up, with edges in place of vertices and triangles in place of edges.

The relevant operator is the edge up-Laplacian $\ul{1}$. Its spectrum can encode structural information about the graph that is not seen by the adjacency or the usual graph Laplacian spectrum. Even among highly regular graphs, the $L_1^{\uparrow}$ spectra can often distinguish nonisomorphic graphs that are indistinguishable by their adjacency and graph Laplacian spectra. In \cite{cioaba-guo-mim-ji}, the authors observed that a family of $3854$ strongly regular graphs with parameters $(35,18,9,9)$ are all distinguished by their $L_1^{\uparrow}$ spectra, even though these graphs have the same adjacency and graph Laplacian spectra; at the same time, \cite{cioaba-guo-mim-ji} also mentions nonisomorphic graphs that are cospectral with respect to higher dimensional Laplacian spectra. The $L_1^{\uparrow}$ spectrum is not a complete invariant. Compared to the adjacency and graph Laplacian spectra, the structural implications of the higher dimensional Laplacian spectra are less well understood. Instead of studying the entire spectrum, we focus on the gap between $0$ and the nonzero eigenvalue that is closest to $0$, or equivalently, the smallest positive eigenvalue.

% This leads to a basic structural question: what structural and combinatorial graph information is encoded by the $L_1^{\uparrow}$ spectrum? Compared to the adjacency and graph Laplacian spectra, the higher dimensional Laplacian spectra are less extensively computed and their structural implications are less well understood. The present paper studies this question from an extremal point of view. Instead of studying the entire spectrum, we focus on the gap between $0$ and the nonzero eigenvalue that is closest to $0$, or equivalently, the smallest positive eigenvalue.

Let $\T$ be a finite nonempty family of $3$-element sets. After fixing an ordering of the underlying point set $V=\bigcup_{T \in \T}T$, one obtains a signed pair-triple incidence matrix $\delta_{1,\T}$; the rows are indexed by the members of $\T,$ the columns are indexed by two element subsets of members of $\T,$ and the entries are the usual boundary signs (see Section \ref{sec:notations-and-background} for definitions). In this paper we study the smallest positive eigenvalue of the corresponding signed incidence Gram matrix:
\begin{gather*}
    \lambda(\T):=\lambda_{\min}^+\left(\delta_{1,\T}^T\delta_{1,\T}\right),
\end{gather*}
as a spectral-combinatorial parameter of the set system $\T$. When $\T$ is the set of all triangles in a graph $G$, then $L_1^{\uparrow}$ is exactly $\delta_{1,\T}^T\delta_{1,\T}$, with possibly additional zero rows and columns indexed by edges not contained in any triangles. The matrix $\delta_{1,\T}^T\delta_{1,\T}$ is positive semidefinite, so, $\lambda(\T)$ is its spectral gap. Changing the total order on $V$ results in a conjugation of $\delta_{1,\T}^T\delta_{1,\T}$ by an orthogonal matrix, hence, $\lambda(\T)$ is independent of the chosen total order on $V.$ Equivalently, if the triples in $\T$ are viewed as triangles of a graph with vertex set $V$, then $\delta_{1,\T}^T\delta_{1,\T}$ is the restricted edge up-Laplacian of the triangle family. Although the definition is spectral, the main results of the paper are extremal-combinatorial: large values of $\lambda(\T)$ force strong local overlap, complete triangle families give exact clique thresholds, and the exact budget function is nonmonotone between consecutive clique sizes.

\paragraph{Local bounds.} Our first theme is local: $\lambda(\T)$ has a strong structural meaning. Let $\Gamma_0(\T)$ denote the support graph of $\T,$ namely the graph formed by all vertices and edges contained in at least one triangle of $\T$, and for each edge $e$ of $\Ga_0(\T)$, denote by $d_e(\T)$ the number of triangles in $\T$ containing $e$. Using the indicator vector of an edge as a test vector, we prove that $\la(\T) \leq d_e(\T)+2$ for every edge $e$ of $\Ga_0(\T)$ (Theorem \ref{thm:large-min-degree-overlap}). Consequently,

\begin{gather*}
    \lceil \la(\T) \rceil-2 \leq \min\{d_e(\T) : e \in E(\Ga_0(\T))\}, \quad \lceil \la(\T) \rceil-1 \leq d_{\min}(\Ga_0(\T)),
\end{gather*}
and both inequalities are sharp for complete graphs on at least $3$ vertices. That is, when $\lambda(\T)$ is large, every edge of $\Gamma_0(\T)$ lies in many triangles in $\T,$ and $\Gamma_0(\T)$ has large minimum degree. This is the analogue of the bound $\mu_2 \leq d_{\min}(G)$ in \eqref{eq:algebraic-connectivity-min-degree-bound}. In particular, every support edge lies in at least $\lceil \lambda(\T) \rceil-2$ triangles and every vertex of $\Gamma_0(\T)$ has degree at least $\lceil \lambda(\T) \rceil-1.$ This is the first theme of the paper.

The second local bound uses vertex links. For a vertex $x$ of $\Ga_0(\T)$, the link of $x$ is the graph on the neighbors of $x$ in $\Ga_0(\T)$ in which $y$ and $z$ are adjacent if and only if $\{x,y,z\} \in \T.$ We prove that
$$
    \la(\T) \leq 1 + \mu
$$
for every nonzero Laplacian eigenvalue $\mu$ of every connected component of every vertex link (Theorem \ref{thm:lambda-link-bound}). The proof lifts an eigenvector of the link component to a test vector supported on the edges at $x$. Comparisons between the Laplacian spectrum of a complex and the spectra of its vertex links go back to Garland \cite{garland}; see also \cite{ballmann-swipolhkatkowski}, \cite{gundert-wagner}, and \cite{lew2025}. In that line of work, the link spectra are typically used to bound the Laplacian spectral gap of the complex from below; here, a single eigenvector of a single link component gives an upper bound. Taking $\mu$ to be the algebraic connectivity of the link component and applying \eqref{eq:algebraic-connectivity-min-degree-bound} to it, we obtain a rigidity statement (Lemma \ref{lemma:link-bound-rigidity}): if an edge $e=\{x,y\}$ lies in exactly $d$ triangles of $\T$ and $\la(\T)>d+1$, then $e$ lies in a complete triple system on $d+2$ points that is a triangle component of $\T$, that is, it shares no edge with the remaining members of $\T$. 

%view $\lambda$ as the edge level analogue of the classical algebraic connectivity of a graph. An important family of problems in extremal spectral graph theory is maximizing the algebraic connectivity given the number of vertices and edges in a graph. Motivated by this problem, 

%In particular,
% We then show that the behavior immediately above this threshold is already nontrivial, and in fact, persists until the next clique threshold $\binom{n+1}{3}$: $\phi(t) \leq n-1$ for $\binom{n}{3} < t < \binom{n+1}{3}.$

\paragraph{Extremal results.} Our second theme is extremal. Motivated by the problem of maximizing the algebraic connectivity, we study the edge-level analogous problem: given a budget of triangles, how large can the smallest positive eigenvalue of $\delta_{1,\T}^T\delta_{1,\T}$ be? For each positive integer $t,$ define:
\begin{gather*}
    \phi(t) = \max_{|\T|=t} \lambda(\T), \qquad \Lambda(t) = \max_{1 \leq s \leq t} \phi(s).
\end{gather*}
Thus $\phi$ and $\Lambda$ play the role of extremal algebraic connectivity functions in the edge-level setting. The natural benchmark is the clique threshold $t=\binom{n}{3}.$ If $\T$ is the full triangle set of $K_n$, then $\lambda(\T)=n.$ We show that this is the best possible for $n \geq 3$: if $|\T|<\binom{n}{3},$ then $\lambda(\T)\leq n-1,$ while if $|\T|=\binom{n}{3}$ and $\lambda(\T) > n-1,$ then $\T$ must in fact be the full triangle set on an $n$-clique. Thus, $\binom{n}{3}$ is the exact threshold for reaching level $n$. 

Our main extremal result is that this threshold is isolated: for every $n \geq 5$,
$$
    \phi(t) \leq n-1 \quad \text{for all } \binom{n}{3} < t < \binom{n+1}{3}
$$
(Theorem \ref{thm:h(n)}).  In particular,
\begin{gather*}
    \phi\left(\binom{n}{3}\right)=n, \quad \phi\left(\binom{n}{3}+1\right) \leq n-1, \quad \phi\left(\binom{n+1}{3}\right)=n+1,
\end{gather*}
so, $\phi$ is not monotone. The proof combines the two local bounds. By the rigidity statements, if $\la(\T)>n-1,$ then every triangle component of $\T$ is either a complete triple system on $n$ points, or has each of its edges in at least $n-1$ of its triangles. A count shows that a component of the second kind has at least $\binom{n+1}{3}$ triangles, and two components of the first kind have $2\binom{n}{3} \geq \binom{n+1}{3}$ triangles when $n \geq 5.$

At the global scale, the clique thresholds determine the monotone envelope $\Lambda$ of $\phi$ exactly:
\begin{gather*}
    \Lambda(t)=\max \left\{n \in \N: \binom{n}{3} \leq t\right\}.
\end{gather*}
and in particular,
\begin{gather*}
    \Lambda(t) = \Theta(t^{\frac{1}{3}}).
\end{gather*}
Thus the largest possible value of $\lambda(\T)$ grows only on the cubic-root scale in the number of triangles. By contrast, the exact-budget extremal function is highly nonmonotone. Nevertheless, we prove that:
\begin{gather*}
    \phi(t) = \Theta(t^{\frac{1}{3}}).
\end{gather*}
The lower bound uses the explicit spectrum of the full triangle family of the graph $K_c \vee \overline{K_b}$ and a Frobenius-type construction.

\paragraph{Contrast with algebraic connectivity} 
There is also a basic contrast to algebraic connectivity: the algebraic connectivity is nondecreasing under edge additions. However, extending $\T$ by adding a triangle can both increase or decrease $\lambda(\T).$ This phenomenon can be observed on $4$ vertices: for
\begin{gather*}
    \T_1=\{\{1,2,3\}\}, \quad
    \T_2=\{\{1,2,3\}, \{1,2,4\}\},\quad
    \T_3=\{\{1,2,3\},\{1,2,4\},\{1,3,4\}\},\\
    \T_4=\{\{1,2,3\},\{1,2,4\},\{2,3,4\},\{1,3,4\}\},
\end{gather*}
we have respectively
\begin{gather*}
    \lambda(\T_1)=3,\quad \lambda(\T_2)=2, \quad \lambda(\T_3)=1, \quad \lambda(\T_4)=4. 
\end{gather*}

\paragraph{Laplacian and incidence matrix viewpoints.} The linear algebraic operators used in this paper are familiar from two standard settings, but here they are used in a different extremal role. Higher-dimensional Laplacians, originating in Eckmann's discrete Hodge theory \cite{eckmann} and appearing in later combinatorial formulations such as Duval--Reiner \cite{duval-reiner}, Horak--Jost \cite{horak-jost}, and Lim \cite{lim}, are usually studied as invariants of a fixed simplicial complex. In that setting, one asks how spectral gaps control expansion, isoperimetry, mixing, or homological vanishing; see, for example, Cheeger-type inequalities and related vanishing results for simplicial complexes \cite{jost-zhang, steenbergen}, as well as Lew's sharp lower bounds under missing-face minimal-degree hypotheses \cite{lew-missing-face}. There is also a parallel linear-algebraic tradition in extremal set theory and design theory, where one studies ranks, $\bmod\, p$ ranks, singular values, or spectra of incidence and inclusion matrices of set systems \cite{frankl-intersection, keevash, grosu-et-al}. Here, the signed edge-triangle incidence matrix is used to define the spectral parameter itself: we vary the triple family $\T$ and study how large $\lambda(\T)$ can be, and what combinatorial properties are forced by a large value of $\lambda(\T).$ 

\paragraph{Relation to the Helmholtzian.}
The parameter $\lambda$ also arises naturally from the total edge Hodge Laplacian:
\begin{gather*}
    L_{1,\T} = \delta_0 \delta_0^T + \delta_{1,\T}^T\delta_{1,\T}.
\end{gather*}
A basic identity shows that 
\begin{gather*}
    \lambda_{\min}^+(L_{1,\T}) = \min \big\{\lambda_{\min}^+(L), \lambda(\T)\big\},
\end{gather*}
where $L$ is the usual graph Laplacian. Thus, once the graph Laplacian is not the bottleneck, the problem of controlling the smallest positive eigenvalue of the total edge Laplacian reduces exactly to controlling $\lambda(\T).$ This gives a direct motivation for treating $\lambda(\T)$ as an object in its own right. If $\T$ is the set of all triangles in $G$, then the matrix $L_{1,\T}$ is precisely the Helmholtzian matrix $\mathcal{H}(G)$ of \cite{lu-lu-et-al-basic} and \cite{lu-lu-et-al-novel}, and the link of a vertex $x$ is the subgraph of $G$ induced by the neighbors of $x$. In particular, \cite{lu-lu-et-al-novel} asks for the relationships between the least or the second least eigenvalue of $\mathcal{H}(G)$ and the structure of $G$. Since $\lambda^+_{\min}(\mathcal{H}(G)) \leq \lambda(\T),$ both local bounds hold for $\lambda^+_{\min}(\mathcal{H}(G))$; in particular, $\lambda^+_{\min}(\mathcal{H}(G)) \leq 1 + \mu$ for every nonzero Laplacian eigenvalue $\mu$ of every connected component of the subgraph induced by the neighbors of any vertex. Thus, the results in Subsections \ref{subsec:large-minimum-degree-subgraphs-triangle-overlaps} and \ref{subsection:local-bound-via-vertex-links} relate the smallest positive eigenvalue of $\mathcal{H}(G)$ to the structure of $G$, and partially address the problem.

Taken together, our results show that $\lambda(\T)$ is a meaningful structural parameter for triangle families: large values force overlap and a support graph of large minimum degree, clique thresholds are rigid, while the associated monotone envelope admits an exact global description. More broadly, this paper illustrates that higher-dimensional Laplacian eigenvalues can encode concrete combinatorial constraints, much as classical spectral parameters do in graph theory.

The paper is organized as follows. Section \ref{sec:notations-and-background} introduces notation and records the linear-algebraic facts used later. In Section \ref{sec:extremality-rigidity-compression}, we present the main combinatorial results. Subsection \ref{subsec:large-minimum-degree-subgraphs-triangle-overlaps} proves the first local bound and its consequences for the support graph. Subsection \ref{subsection:local-bound-via-vertex-links} proves the link bound and the rigidity lemma. Subsection \ref{subsec:rigidity-extremal-constructions} establishes the exact clique-threshold rigidity and the exact formula for $\Lambda.$ Subsection~\ref{subsection:phi-nonmonotonicity} proves that $\phi(t) \leq n-1$ between consecutive clique thresholds, and hence the nonmonotonicity of $\phi$. In subsection \ref{subsec:spectrum-glued-cliques}, we give an explicit description of the clique $2$-down Laplacian of the graph $K_c \vee \overline{K_b}$ constructed by joining $b$ disjoint vertices to each vertex in a copy of $K_c.$ In subsection \ref{subsec:growth-rate-phi}, we use this explicit value along with a Frobenius-type construction to derive the exact growth rate of $\phi.$ Section \ref{sec:conclusion} concludes with some further questions about the finer structure of extremal and near-extremal triangle families.

\section{Notations and Background}
\label{sec:notations-and-background}
In this section, we introduce notations and terminology used throughout the paper, and recall some linear algebraic facts that will be useful to us in later sections. A graph $G=(V,E)$ refers to a finite, simple, undirected graph, with an arbitrary but fixed total order on $V$. For an edge $e=\{x,y\}$ and a triangle $T=\{a,b,c\}$ in $G$, we define the sign of $e$ in $T,$ denoted $[T:e]$ as follows:
\begin{gather*}
    [T:e] :=  \begin{cases}
        0 & \text{ if } \{x,y\} \not \subset \{a,b,c\},\\
        1 & \text{ if } \{a,b,c\} \setminus \{x,y\} = \{\max\{a,b,c\}\} \text{ or } \{a,b,c\} \setminus \{x,y\} = \{\min\{a,b,c\}\},\\
        -1 & \text{ otherwise}
    \end{cases}.
\end{gather*}
Similarly, if $x$ is a vertex in $G$, the sign of the vertex $x$ in $e$, denoted $[e:x]$, is defined to be $0$ if $e$ is not incident with $x$, $1$ if $x$ is the larger of the two vertices incident with $e$, and $-1$ otherwise.

We denote by $\delta_0$ the $E \times V$ signed incidence matrix with $(\delta_0)_{e,v}=[e:v]$ for every vertex $v$ and edge $e$ in $G$. Then, $L:=L_0^{\uparrow}:=\delta_0^T \delta_0$ is the usual graph Laplacian. Let $G$ be a graph and let $\T$ be a nonempty set of triangles in $G$. We may view $\{\emptyset\} \cup V(G) \cup E(G) \cup \T$ as a $2$-dimensional abstract simplicial complex. We denote by $\delta_{1,G,\T}$ the usual $|\T| \times |E|$ coboundary map $\re^E \to \re^{\T}.$ That is, if $e$ is an edge in $G$ and $T \in \T$ is a triangle, then, $\left(\delta_{1,G,\T}\right)_{T,e} = [T:e]$. More generally, if $v \in \re^{\T}, w \in \re^{E(G)}$ and $e \in E(G), T \in \T,$ we denote by $[v:T]$ the value of $v$ at $T$ and by $[w:e]$ the value of $w$ at $e$. If $G$ is understood from the context, we may suppress the subscript $G$ and write $\delta_{1,\T}$ instead. We similarly define $\ul{1, G, \T}$ or $\ul{1,\T}$ as the matrix $\delta_{1,G,\T}^T \delta_{1,G,\T},$ and analogously, $\dl{1,\T}=\delta_0\delta_0^T,\dl{2,G,\T}=\delta_{1,G,\T}\delta_{1,G,\T}^T.$ When $\rankg{}{\ul{1,\T}} > 1$, we will also denote by $\tau(\T)$ the second smallest positive eigenvalue of $\ul{1,\T}$, allowing for the possibility $\tau(\T)=\lambda(\T).$ We reserve the symbol $\delta_{1,G}$ or $\delta_1$ for the operator $\delta_{1, G, X_2(\cl{G})},$ i.e., the ``full'' coboundary operator on the graph $G.$ We also define $\ul{1,G,\emptyset}=0_{|E| \times |E|}$. Finally, we define $\lambda(G,\T)=\lambda(\T)$ to be the smallest positive eigenvalue of $\ul{1,G,\T}$ or the smallest positive eigenvalue of $\dl{2,G,\T}.$\par 

%Furthermore, if $T$ is a triangle in $G$ and $1_T$ its characteristic vector, then the vector $\delta_{1,\T}^T(T):=\delta_{1,\T}^T(1_T)$ is a vector of length $|E(G)|$ that is independent of the set of triangles $\T.$ Thus, the expression $\delta_1^T(T)$ is defined unambiguously for a given graph $G$ and a triangle $T$ in $G$.

All (co)boundary operators considered here are over $\re,$ and we assume a global ordering on the vertex set of $G,$ which we will often take as the set $[n]=\{1,\dots,n\}$ for an $n$-vertex set. Furthermore, $\ul{1,\T}$ and $\dl{2,\T}$ will be assumed to have their rows and columns ordered by the lexicographic order on the set of edges in $G$, and the set of triangles in $\T$, respectively. If $e$ is an edge in $G$, we write $\delta_{1,\T}1_e \in \re^{\T}$ for the vector $\delta_11_e,$ where $1_e \in \re^E$ is the characteristic vector of the edge $e$.\par

\subsection{Linear Algebraic Background}
We recall some well-known properties of these matrices in the following theorem. 
\begin{theorem} \label{thm:hodge}
    Let $G$ be a graph and let $\T$ be a nonempty set of triangles in $G$. Let $\delta_0, \delta_{1,\T}, \ul{1,\T}, \dl{2,\T}$ be defined above. Then,
    \begin{enumerate}[(a)]
        \item The matrices $\ul{0},\ \ul{1,\T},\ \dl{1, \T}, \ \dl{2,\T}$ are positive semidefinite, and $\ul{0}$ is the usual graph Laplacian. 
        \item $\delta_{1,\T}\delta_0 = 0, \ \delta_0^T\delta_{1,\T}^T = 0,\ \ul{1,\T}\dl{1,\T}=0, \ \dl{1,\T}\ul{1,\T}=0$.
        \item $\re^E$ orthogonally decomposes as:
        \begin{gather*}
            \re^E = \im \delta_0 \oplus \im \delta_{1,\T}^T \oplus (\ker \delta_0^T \cap \ker \delta_{1,\T}) = \im \dl{1,\T} \oplus \im \ul{1,\T} \oplus (\ker \delta_0^T \cap \ker \delta_{1,\T}),
        \end{gather*}
        and
        \begin{gather*}
            \im \delta_0 = \im \dl{1,\T}, \quad \im \delta_{1,\T}^T = \im \ul{1,\T}, \quad \ker(L_1) = \ker(\delta_0\delta_0^T + \delta_{1,\T}^T\delta_{1,\T}) = \ker \delta_0^T \cap \ker \delta_{1,\T}.
        \end{gather*}
        \item The nonzero parts of the spectra of $\ul{1,\T}$ and $\dl{2,\T}$ are the same.
    \end{enumerate}
\end{theorem}
\begin{proof}
    We prove the first part of (b) and give appropriate references to the remaining parts.
    \begin{enumerate}[(a)]
        \item The part (a) is immediate since each of the matrices $\ul{0},\ \ul{1,\T},\ \dl{1, \T}, \ \dl{2,\T}$ is a Gram matrix.
        \item Recall that $\delta_1$ denotes the full edge-triangle coboundary map in $G.$ It is well known that $\delta_1 \delta_0 = 0$ [equation (3.1) in \cite{duval-reiner} or Lemma 2.1 in \cite{hatcherat}]. That is, if $v$ is a vertex and $T$ a triangle in $G$, then,
        \begin{gather*}
            \left(\delta_1\delta_0\right)_{T,v} = \sum_{e \in E(G)}[T:e]\cdot [e:v] = 0.
        \end{gather*}
        However, by definition, the matrix $\delta_{1,\T}$ is obtained by removing the rows indexed by triangles in $G$ not in $\T$. Thus, if $T \in \T$, then 
        \begin{gather*}
            \left(\delta_{1,\T}\delta_0\right)_{T,v} = \left(\delta_{1}\delta_0\right)_{T,v} = \sum_{e \in E(G)}[T:e]\cdot [e:v]=0.
        \end{gather*}
        This proves that $\delta_{1,\T}\delta_0=0.$ By transposing both sides of the equation, we obtain $\delta_0^T \delta_{1,\T}^T=0.$ The remaining two claims follow by expanding $\ul{1,\T}\dl{1,\T}$ and $\dl{1,\T}\ul{1,\T}$ as $\delta_{1,\T}^T\delta_{1,\T}\delta_{0}\delta_{0}^T$ and $\delta_{0}\delta_{0}^T\delta_{1,\T}^T\delta_{1,\T}$, respectively.
        \item Since $\delta_1 \delta_0=0,$ the first decomposition follows using equation (9) of Theorem 5.2 in \cite{lim}, the first part of the second claim and the second decomposition follows by Theorem 5.1 in \cite{lim}, and the last part of the second claim by part (6) of Theorem 5.2 in \cite{lim}.
        \item It follows from the general fact that if $A,B$ are two matrices such that both $AB$ and $BA$ are defined, then the nonzero parts of the spectra of $AB$ and $BA$ are the same.
    \end{enumerate}
\end{proof}

We now precisely state and prove the min-gap identity that demonstrates that studying $\lambda(\T)$ is the right lens through which to view the optimization problem of $\lambda_{\min}^+(L_{1,\T})$. A similar result in the context of clique complexes of graphs is folklore in the Hodge spectral theory, see for example the discussions following Theorem 3.3 in \cite{duval-reiner}. We provide a self-contained proof in our general setting.

 \begin{prop} \label{prop:lambda-min-identity}
     Let $G$ be a graph, $\T$ a nonempty set of triangles in $G$. Let $L$ denote the usual graph Laplacian and $L_1$ the restricted edge total-Laplacian $L_{1,\T}$. Then,
     \begin{gather*}
         \lambda_{\min}^+(L_1) = \min \big\{\lambda_{\min}^+(L), \lambda(\T)\big\}.
     \end{gather*}
     In particular, if $G$ is connected, then,
     \begin{gather*}
         \lambda_{\min}^+(L_1) = \min \big\{\lambda_2(L), \lambda(\T)\big\}.
     \end{gather*}
 \end{prop}
 \begin{proof}
     It suffices to prove that
     \begin{gather}
         \lambda_{\min}^+(L_1) = \min \big\{\lambda_{\min}^+(\dl{1,\T}), \lambda(\T)\big\},
     \end{gather}
     since the nonzero part of the spectra of $\dl{1}=\delta_{0}\delta_0^T$ is the same as the nonzero part of the spectra of $L=\delta_0^T \delta_0$. Since $L_1$ is a PSD matrix, $\lambda_{\min}^+(L_1)$ is the smallest nonzero eigenvalue of $L_1.$ Since $L_1$ is symmetric, by the Courant-Fischer Theorem,
     \begin{align*}
         \lambda_{\min}^+(L_1) &= \min_{x \perp \ker(L_1),\ \|x\|=1} x^T L_1 x\\
         &= \min_{\substack{f \in \im \delta_0, \ g\in \im \delta_{1,\T}^T,\\ \|f\|^2+\|g\|^2=1}} (f+g)^T(\delta_0\delta_0^T + \delta_{1,\T}^T \delta_{1,\T})(f+g) \tag{by Theorem \ref{thm:hodge} (c)}\\
         &= \min_{\substack{f \in \im \delta_0, \ g\in \im \delta_{1,\T}^T,\\ \|f\|^2+\|g\|^2=1}} f^T(\delta_0\delta_0^T)f + g^T(\delta_{1,\T}^T \delta_{1,\T})g \tag{\text{Theorem } \ref{thm:hodge} (b)}\\
         &= \min_{0 \leq \alpha, \beta \leq 1,\  \alpha^2+\beta^2=1} \Big( \min_{f\in \im \delta_0, \ \|f\|=\alpha} f^T\dl{1,\T}f \ +\  \min_{g\in \im \delta_{1,\T}^T, \ \|g\|=\beta} g^T\ul{1,\T}g\Big)
     \end{align*}
     However, since
     \begin{gather*}
        (\ker \dl{1,\T})^{\perp} =  (\ker \delta_0^T)^{\perp} = \im \delta_0, \text{ and,}\\
        (\ker \ul{1,\T})^{\perp} =  (\ker \delta_{1,\T})^{\perp} = \im \delta_{1,\T}^T,
     \end{gather*}
     the last quantity above is equal to:
     \begin{align*}
         & \min_{0 \leq \alpha, \beta \leq 1,\  \alpha^2+\beta^2=1} \Big( \alpha^2 \lambda_{\min}^+(\dl{1,\T}) + \beta^2 \lambda_{\min}^+(\ul{1,\T})\Big) = \min \big\{\lambda_{\min}^+(\dl{1,\T}), \lambda(\T)\big\}. 
     \end{align*}
     This completes the proof.
 \end{proof}

\subsection{\texorpdfstring{The $\lambda$ Function}{The lambda function}}
In this short subsection, we formalize the notion of the $\lambda$ function. Let $\T$ be a nonempty finite set of $3$-element subsets of a totally ordered set $X.$ If $G$ and $G'$ are two graphs containing the triangles in $\T,$ then,
\begin{gather*}
    \lambda_{\min}^+(\ul{1,G,\T})=\lambda_{\min}^+(\dl{2,G,\T})=\lambda_{\min}^+(\dl{2,G',\T})=\lambda_{\min}^+(\ul{1,G',\T}).
\end{gather*}
Here, the middle two terms are equal since the matrices $\dl{2,G,\T}$ and $\dl{2,G',\T}$ only depend on how the triangles in $\T$ share edges, and are otherwise independent of the supporting graphs $G$ and $G'.$

\begin{defn}[Support graph]
    Let $\T$ be a nonempty finite set of $3$-element subsets of a totally ordered set $X.$ The support graph of $\T$, denoted $\Gamma_0(\T)$, is the graph whose vertex set is $\bigcup_{T \in \T}T$, and two distinct vertices $x,y$ are adjacent if and only if there is a $T \in \T$ with $x, y \in T.$
\end{defn}

In particular, every vertex and every edge in $\Gamma_0(\T)$ is contained in at least one triangle in $\T.$ The above discussion motivates the following definition.

\begin{defn}
    Let $\T$ be a nonempty finite set of $3$-element subsets of a totally ordered set $X,$ and let $\Gamma_0(\T)$ be the support graph for $\T.$ We define $\lambda(\T):=\lambda_{\min}^+(\ul{1,\T}).$
\end{defn}
% \begin{prop}
%     Let $G$ be a graph, $\T$ a nonempty set of triangles in $G$. Then,
%     \begin{gather*}
%         \lambda(\T)=\lambda_{\min}^+(\ul{1,\T}).
%     \end{gather*}
% \end{prop}

Finally, we note that $\lambda(\T)=\lambda_{\min}^+(\dl{2,\Gamma_0(\T), \T})$ is independent of the ordering of the vertex set of $V(\Gamma_0(\T)) = \bigcup_{T \in \T}T$, because changing the total order on $V$ results in conjugating $\dl{2,\Gamma_0(\T), \T}$ by a diagonal orthogonal matrix (see Lemma A1 in \cite{cioaba-guo-mim-ji}). Thus, we may view $\lambda$ as a spectral parameter of finite families of $3$-element subsets, independent of the total order on the vertex set and of the host graph.

This gives the motivation for treating $\lambda(\T)$ as a parameter of the triangle family $\T$ itself, not the underlying graph.

\section{Extremality and Rigidity}
\label{sec:extremality-rigidity-compression}
\subsection{Large Minimum Degree Subgraphs and Triangle Overlaps}
\label{subsec:large-minimum-degree-subgraphs-triangle-overlaps}

% \begin{defn}[Triangle-intersection graph]
%     Let $\T$ be nonempty finite set of $3$-element subsets of a totally ordered set $X.$ The triangle-intersection graph of $\T$, denoted $\Gamma_1(\T),$ is the graph whose vertex set is $\T$, and two vertices $T_1,T_2$ are adjacent if and only if $|T_1 \cap T_2|=2.$
% \end{defn}

\begin{defn}
    Let $G$ be a graph. We denote by $d_{\min}(G)$ the minimum vertex degree in $G:$
    \begin{gather*}
        d_{\min}(G) = \min_{v \in V(G)} d_G(v).
    \end{gather*}
\end{defn}

\begin{theorem} \label{thm:large-min-degree-overlap}
    Let $\T$ be a nonempty finite set of $3$-element subsets of some totally ordered set $X$, and let $n$ be a positive integer such that $n \leq \lceil\lambda(\T)\rceil.$ Let $G$ denote $\Gamma_0(\T).$ Then,
    \begin{enumerate} [(a)]
        \item for every edge $e=\{x,y\}$ in $G$, there are at least $(n-2)$ triangles in $\T$ containing the edge $e,$
        \item for every edge $e=\{x,y\}$ in $G$, we have $|N_G(x) \cap N_G(y)| \geq (n-2),$
        \item for every vertex $x$ in $G,$ there are at least $(n-2)$ triangles in $\T$ containing $x$ as a vertex, and $d_G(x) \geq (n-1)$ and,
        \item $d_{\min}(G) \geq (n-1)$ and $|V(G)|\geq n.$
    \end{enumerate}
\end{theorem}

\begin{proof}
    The statements are immediate for $n=1$ and $n=2$. We assume $n \geq 3$. Let $1_e \in \re^{E(G)}$ be the indicator vector of $e$. Let $0 < d_e = |\{T \in \T: e \text{ is an edge of } T\}|$, and $g=\delta_{1,\T}(1_e).$ Then,
    \begin{gather*}
        g = \delta_{1,\T}(1_e) \in \im \delta_{1,\T}=\im \dl{2,\T} = \left(\ker \dl{2,\T}\right)^{\perp}.
    \end{gather*}
    Thus,
    \begin{gather*}
        \lambda(\T) = \lambda_{\min}^+(\ul{1,\T}) = \lambda_{\min}^+(\dl{2,\T}) \leq \frac{\langle g, \dl{2,\T}g \rangle }{\|g\|^2} = \frac{\|\delta_{1,\T}^Tg\|^2}{\|g\|^2}. 
    \end{gather*}
    Now, $g$ has exactly $d_e$ nonzero coordinates, each $\pm 1.$ So, $\|g\|^2=d_e.$ On the other hand, the triangles in $\T$ containing the edge $e$ share the edge $e,$ and each has a vertex and two edges not contained in any other triangle in $\T$ containing the edge $e.$ Thus, $(\delta_{1,\T}^Tg)_e = \sum_{e \subset T \in \T}[T:e]^2 = d_e,$ and, $\delta_{1,\T}^Tg$ has exactly $2d_e$ additional nonzero coordinates, each with entries $\pm 1.$ Consequently, $\|\delta_{1,\T}^Tg\|^2 = d_e^2+2d_e.$ We thus obtain:
    \begin{gather*}
        \lambda(\T) \leq \frac{\|\delta_{1,\T}^Tg\|^2}{\|g\|^2} = \frac{d_e^2+2d_e}{d_e}=d_e+2. 
    \end{gather*}
    Since $d_e$ is a positive integer, we conclude that $n \leq \lceil \lambda(\T) \rceil \leq d_e + 2.$ This proves part (a). Part (b) follows from part (a) by observing that any edge $e$ in $G$ is contained in at least $(n-2)$ triangles in $\T$, and the third vertices of these triangles are pairwise distinct, and are common neighbors of $x$ and $y$. To prove part (c), note that $x$ is contained in at least one edge in $G$, say, $e.$ Then, by part (a), there are at least $(n-2)$ triangles in $\T$ containing the edge $e$, and each of these triangles also contain the vertex $x.$ Since two distinct triangles in $G$ can share at most one edge, each of these triangles contains an edge distinct from $e$ that is incident with $x.$ The other endpoints of these edges are pairwise distinct, and distinct from $x$. Thus, $d_G(x) \geq (n-2)+1=n-1$. This proves part (c). Part (d) follows from part (c).
\end{proof}

Theorem \ref{thm:large-min-degree-overlap} is the main local structural engine of this paper. It shows that a large spectral gap $\lambda(\T)$ forces every support edge to have large codegree, and therefore prevents a sparse or weakly overlapping support graph. In what follows, $\T$ will be assumed to be a nonempty finite set of $3$-element subsets of a totally ordered set $X$, unless otherwise stated.

\begin{theorem} \label{thm:structural-theorem-1}
    Let $G$ be a graph, $\T$ a nonempty set of triangles in $G,$ and $\lambda = \lambda(\T).$ Then, $G$ has a subgraph $H$, namely $\Gamma_0(\T),$ such that $|V(H)|\geq \lceil \lambda \rceil$, $d_{\min}(H)\geq \lceil \lambda \rceil-1,$ and every edge in $H$ is contained in at least $(\lceil \lambda \rceil-2)$ triangles in $H$. 
\end{theorem}

\begin{proof}
    Let $H=\Gamma_0(\T).$ Then, by Theorem \ref{thm:large-min-degree-overlap}, $H$ is a subgraph of $G$ with the desired properties.
\end{proof}

\begin{corollary} \label{thm:structural-theorem-2}
    Let $G$ be a graph with at least one triangle and let $\ul{1}$ be the edge up-Laplacian of the clique complex of $G.$ Denote $\lambda = \lambda_{\min}^+(\ul{1}(G)).$ Let $H$ be the subgraph of $G$ consisting of the vertices and edges of $G$ that are contained in at least one triangle in $G$. Then, $|V(H)|\geq \lceil \lambda \rceil$, $d_{\min}(H) \geq \lceil \lambda \rceil-1,$ and every edge in $H$ is contained in at least $(\lceil \lambda \rceil-2)$ triangles in $H$.
\end{corollary}

\begin{proof}
    Let $\T$ be the set of all triangles in a graph $G.$ Then, $\lambda(\T)=\lambda_{\min}^+(\ul{1}(G))= \lambda,$ and so, $\Gamma_0(\T)=H$ satisfies the desired properties by Theorem \ref{thm:structural-theorem-1}.
\end{proof}

By Proposition \ref{prop:lambda-min-identity}, the corollary remains true with $\la$ replaced by $\la_{\min}^+(\mathcal{H}(G))$; compare \cite[Theorem 3.2.1]{lu-lu-et-al-novel}.

\subsection{A Local Bound via Vertex Links} \label{subsection:local-bound-via-vertex-links}
In the previous subsection, we derived an upper bound on $\la(\T)$ using the indicator vector of an edge. In this section, we use a vertex link based test-vector to derive a more refined upper bound on $\la(\T).$ The method is motivated by the seminal work of Garland (see \cite{garland}). See also \cite{gundert-wagner},\cite{ballmann-swipolhkatkowski}, and more recently, \cite{lew2025}.

\begin{defn} \label{defn:link}
    Let $\T$ be an arbitrary triangle set and $G=\Gamma_0(\T)$ be its support graph. For $x \in V(G)$, we define the link of $x$ in $G$ with respect to $\T$, denoted $\lk{x,\T}$, or simply $\lk{x}$, as the subgraph of $G$ on vertex set $N_G(x)$ and with edges of the form $\{y,z\}$ such that $\{x,y,z\} \in \T$. If $\{x,y,z\} \in \T$ for vertices $x,y,z$, then we denote by $C_x(y)$ the connected component of $\lk{x}$ containing the vertex $y$.
\end{defn}

\begin{defn}
    Let $\T$ be a nonempty set of triangles. We define the triangle-graph $\Gamma_1(\T)$ to be the graph on vertex set $\T$ with the edges defined as follows: two vertices $T_1, T_2 \in \T$ are adjacent if and only if $T_1, T_2$ shares exactly two vertices. A triangle component of $\T$ is a connected component of $\Gamma_1(\T).$
\end{defn}

\begin{defn}
    If $G$ is any graph with at least one triangle, a triangle component of $G$ is a triangle-component of $\T$, where $\T$ is the set of all triangles in $G$.
\end{defn}

\begin{defn}
    Let $\T$ be a nonempty set of triangles. If $e \in E(\Ga_0(\T))$, then we denote
    $$
        d_e(\T)=|\{T \in \T : e \subset T\}|.
    $$
\end{defn}

Note if $e$ is an edge of $E(\Ga_0(\T))$, then all triangles in $\T$ containing the edge $e$ belongs to the same triangle component of $\T$. In particular, if $e$ is an edge contained in a triangle in $\T,$ then there is a unique triangle component $\T'$ of $\T$ with a triangle containing the edge $e$, and $d_e(\T')=d_e(\T).$ 

\begin{defn}
    If $\T$ and $\T'$ are two nonempty sets of triangles, then we say $\T$ and $\T'$ are isomorphic if there is a bijections $F: V(\Ga_0(\T)) \to V(\Ga_0(\T'))$ such that
    $$
        \T' = \{\{F(x), F(y), F(z)\} : \{x,y,z\} \in \T\}.
    $$
\end{defn}

Note that for arbitrary triangle family $\T,$ a triangle component in $\T$ may be a proper subset of a  triangle component of $\Gamma_0(\T).$ 

 Let $\T$ be an arbitrary nonempty family of triangles and $G=\Gamma_0(\T)$ be the support graph of $\T.$
 If $e=\{x,y\} \in E(G)$ and $C$ is the connected component of $\lk{x}$ containing $y$, then for every $T=\{x,z,w\} \in \T$ with $z \in C$ satisfies $w \in C$. That is, every $T \in \T$ with $x \in T$ is supported on $\{x\}\cup C$ for some connected component $C$ of $\lk{x}$. Also, if $C$ is a connected component of $\lk{x}$ that is a complete graph, then every $3$-set of the form $\{x,y,z\}$ with $y,z \in C$ is in $\T$. Consequently, in this case, for every $z \in C$, $\{w \in V(G) : \{x,z,w\} \in \T\} = C-z$ holds.

Let $x \in V(G)$, and $C$ be a connected component of $\lk{x}$, and $H=\lk{x}[C]$. Denote by $L=L_{x,C}$ the usual graph Laplacian of $H$. In particular, if $x \in V(G)$, $y \sim x$ and $C$ is the connected component of $\lk{x}$ containing $y$, then for any $f \in \re^{C}$, we have
$$
(Lf)(y) = |\{z \in C: \{x,y,z\} \in \T\}|\cdot f(y) - \sum_{z \in C, \{x,y,z\} \in \T} f(z)=\sum_{y,z \in C, \{x,y,z\} \in \T} (f(y)-f(z)).
$$

\begin{theorem} \label{thm:lambda-link-bound}
    Let $\T$ be an arbitrary nonempty triangle set and $G=\Gamma_0(\T)$. Let $x \in V(G)$, and $C$ be a connected component of $\lk{x}$, $H=\lk{x}[C]$, $L=L_{x,C}$ the usual graph Laplacian of $H$. Let $\mu$ be a nonzero Laplacian eigenvalue of $L$ with an eigenvector $c$. Then,
    \begin{gather*}
        \lambda(\T) \leq 1 + \mu.
    \end{gather*}
\end{theorem}

\begin{proof}
    We have $c^TLc=\mu\|c\|^2$. We construct two functions below:
$$
f \in \re^{E(G)}, \quad f(e) = \begin{cases} [\{x,y\}:x] \cdot c(y), & y\in C, y \sim x\\
0, & \text{otherwise}\end{cases},
$$
and
$$
g \in \re^{\T}, \quad g = \delta_{1,\T}f = \delta_1f.
$$
Then, $g$ is supported on triangles of the form $\{x,y,z\} \in \T,$ where $y,z \in C$, by the definition of $C$. Call the set of such triangles $\T' \subset \T$, and fix such a triangle $T =\{x,y,z\}$. Then,
\begin{align*}
g(T) &= [\{x,y,z\}:\{x,y\}] \cdot f(\{x,y\}) + [\{x,y,z\}:\{x,z\}] \cdot f(\{x,z\})\\
&= [\{x,y,z\}:\{x,y\}] \cdot [\{x,y\}:x] \cdot c(y) + [\{x,y,z\}:\{x,z\}] \cdot [\{x,z\}:x] \cdot c(z)\\
&= [T:\{x,y\}][\{x,y\}:x] \cdot (c(y)-c(z)).
\end{align*}
Thus,
$$
\|g\|^2 = \sum_{T \in \T'} g(T)^2= \sum_{\{x,y,z\} \in \T'} (c(y)-c(z))^2 = c^T L c= \mu \|c\|^2 > 0.
$$
We now compute $\|\delta_1^Tg\|^2$. Note that $\delta_1^Tg$ is supported on edges of the form $\{x,y\}$ or $\{y,z\}$ with $y,z \in C$. If $e$ is an edge of latter form, then,
$$
\|(\delta_1^Tg)(e)\|^2=(c(y)-c(z))^2.
$$
On the other hand, if $e=\{x,y\}$, then the total contribution to $e$ in $\delta_1^Tg$ is:
\begin{gather*}
    \sum_{\{x,y,z\} \in \T} [\{x,y,z\}:\{x,y\}]^2 \cdot [\{x,y\}:x] (c(y)-c(z))\\
=\sum_{\{x,y,z\} \in \T'} [\{x,y,z\}:\{x,y\}]^2 \cdot [\{x,y\}:x] (c(y)-c(z))\\
= [\{x,y\}:x]\sum_{\{x,y,z\} \in \T'} (c(y)-c(z))\\
=[\{x,y\}:x](Lc)(y)=\mu [\{x,y\}:x]c(y)=\mu f(\{x,y\}).
\end{gather*}

That is,
$$
\|\delta_1^Tg\|^2=\mu^2 \|c\|^2 + c^TLc= \mu^2 \|c\|^2 + \mu\|c\|^2.
$$
Thus,
$$
\lambda(\T) \leq \frac{\|\delta_1^Tg\|^2}{\|g\|^2}= \frac{\mu(\mu+1)\|c\|^2}{\mu\|c\|^2}=\mu+1,
$$
as needed.
\end{proof}

% \begin{corollary}\label{cor:link-bound-local-bounds}
%     In the above setting, for any $y \in C$,
%     \begin{enumerate}
%         \item if $\lk{x}[C]$ is a complete graph, then
%         $$
%         |C| \geq \ld (\T)-1, \quad d_{\min}(\lk{x}[C]) \geq |C|-1 \geq \ld (\T)-2,
%         $$
%         thus,
%         \begin{gather*}
%             \ld(\T) \leq d_{\min}(\lk{x}[C])+2, \quad \ld(\T) \leq |C|+1,\\
%             |\{T \in \T: \{x,y\} \subset T\}|=|\{z \in C : \{x,y,z\} \subset \T\}| \geq \ld(\T)-2,
%         \end{gather*}
%         \item if $\lk{x}[C]$ is not a complete graph, then
%         $$
%         |C| \geq \ld(\T) +1, \quad d_{\min}(\lk{x}[C]) \geq \ld(\T)-1,
%         $$
%         thus,
%         \begin{gather*}
%             \la(\T) \leq d_{\min}(\lk{x}[C])+1, \quad \la(\T)\leq |C|-1,\\
%             |\{T \in \T: \{x,y\} \subset T\}| =|\{z \in C : \{x,y,z\} \subset \T\}| \geq \la(\T)-1.
%         \end{gather*}
%     \end{enumerate}
% \end{corollary}

By taking $\mu$ to be the smallest positive eigenvalue of $L$ or the algebraic connectivity of $\lk{x}[C]$ and using the \eqref{eq:algebraic-connectivity-min-degree-bound} bound for non-complete graphs, we obtain the following corollaries. For both corollaries below, in the non-complete cases, one considers the fact that if $C$ is a connected component of $\lk{x}$, then for $y \in C$ and with $H=\lk{x}[C],$:
    $$
        d_H(y)=|\{z \in C : \{x,y,z\} \in \T\}|=d_{\{x,y\}}(\T).
    $$

\begin{corollary}
    Let $\T$ be a nonempty set of triangles and $e=\{x,y\}$ is an edge of $G=\Ga_0(\T).$ Let $H=\lk{x}[C_x(y)]$ and let $a$ be the algebraic connectivity of $H$. Then,
    \begin{gather*}
        \la(\T) \leq 1 + a.
    \end{gather*}
    Consequently, if $H$ is a complete graph, then $|V(H)|\geq \la(\T)-1$ and $d_{\min}(H) \geq \la(\T)-2,$ and if $H$ is not a complete graph, then $d_{\min}(H) \geq \la(\T)-1.$
\end{corollary}

\begin{corollary}\label{cor:link-bound-local-bounds}
    Let $\T$ be a nonempty set of triangles and $x$ is a vertex of $G=\Ga_0(\T).$ If $C$ is a connected component of $\lk{x}$, then for any $y \in C$,
    \begin{enumerate}
        \item if $\lk{x}[C]$ is a complete graph, then
        \begin{gather*}
           |\{T \in \T: \{x,y\} \subset T\}|=|\{z \in C : \{x,y,z\} \in \T\}| \geq \ld(\T)-2,
        \end{gather*}
        \item if $\lk{x}[C]$ is not a complete graph, then
        \begin{gather*}
            |\{T \in \T: \{x,y\} \subset T\}| =|\{z \in C : \{x,y,z\} \in \T\}| \geq \la(\T)-1.
        \end{gather*}
    \end{enumerate}
\end{corollary}

When $\T$ is the set of all triangles in $G$, Theorem \ref{thm:lambda-link-bound} gives $\lambda_{\min}^+(\mathcal{H}(G)) \leq 1 + \mu.$

% \begin{rk} \label{rk:link-bound-rigidity}
%     In the above context, let $e=\{x,y\} \in E(G)$ be such that $d_{\T}(e)=d$ and $\la(\T) > d+1$. Then, by the above dichotomy, if $C$ is the connected component of $\lk{x}$ containing $y$, then $\lk{x}[C]$ is a complete graph on $(d+1)$ vertices, and in particular, for every $z,w \in C, z \ne w$, $\{x,z,w\} \in \T$. Furthermore, for $z \in C$, $(C-z)\cup \{x\}$ is connected component of $\lk{z}$ containing $x$ that is a complete graph. In particular, $\binom{C \cup \{x\}}{3}$ is a triangle component of $G$.
% \end{rk}

\begin{lemma}\label{lemma:link-bound-rigidity}
    Let $\T$ be an arbitrary nonempty family, $G=\Ga_0(\T),$ and $e=\{x,y\} \in E(G).$ Suppose $d_e(\T)=d$ and $\la(\T)>d+1$. If $C=C_x(y),$ then $|C|=d+1$, and $\binom{\{x\} \cup C}{3}$ is a triangle component of $\T.$
\end{lemma}

\begin{proof}
    Since $d_{\{x,y\}}(\T)< \la(\T)-1$, by the dichotomy in Corollary \ref{cor:link-bound-local-bounds}, $\lk{x}[C]$ is a complete graph, that is,
    $$
        \{z \in V(G): \{x,y,z\} \in \T \} = \{z \in C: \{x,y,z\} \in \T \} = C-y,
    $$
    and $|C|=d+1$. For $z \in C-y,$ 
    $$
        \{w \in V(G): \{x,z,w\} \in \T\} = \{w \in C: \{x,z,w\} \in \T\} \subset C-z.
    $$
    Since $\la(\T)>d+1$, by Theorem \ref{thm:large-min-degree-overlap}, we have $d_{e} \geq d$ for every edge $e$ of $G$. Since $|C|=d+1$, we conclude that
    $$
        \{w \in V(G): \{x,z,w\} \in \T\} = \{w \in C: \{x,z,w\} \in \T\}= C-z.
    $$
    That is, $d_{\{x,z\}}=d<\la(\T)-1.$ Again, by the dichotomy in Corollary \ref{cor:link-bound-local-bounds}, $\lk{z}[C_z(x)]$ is a complete graph, and so, that every triangle of the form $\{z,p,w\}$ with $p,w \in (C-z)\cup \{x\}, p \ne w$ is in $\T.$ This shows that $C_z(x)=(C-z)\cup \{x\},$ and $\binom{\{x\} \cup C}{3} \subset \T.$

    Finally, to show that $\binom{\{x\} \cup C}{3}$ is a triangle component, it suffices to show the following claim: if $\{z,w,p\} \in \T$ with $z,w \in C$, then $p \in \{x\} \cup C.$ Since $\binom{\{x\} \cup C}{3} \subset \T$, both the triangles $\{x,z,w\}$ and $\{z,w,p\}$ are in $\T$ (possibly the same triangle), and so, $x,p$ are in the same connected component of $\lk{z}$. However, $C_z(x)=(C-z)\cup\{x\}$, and so, $p \in \{x\} \cup C.$
\end{proof}

\begin{corollary}\label{cor:link-bound-}
    Let $\T$ be an arbitrary nonempty family, $G=\Ga_0(\T)$ and $n \geq 3$ is a positive integer. Let $\la(\T) > n-1$, and let $\T'$ a triangle component of $\T$. Then, either $\T'$ is isomorphic to $\binom{[n]}{3}$, or every edge $e$ contained in a triangle $T \in \T'$ is contained in at least $(n-1)$ triangles in $\T'$.
\end{corollary}

\begin{proof}
    Let $e=\{x,y\}$ be an edge contained in a triangle in $\T'$, and assume that $d_e(\T')=d_e(\T) \leq n-2.$ Denoting $d=d_e(\T')$, $\la(\T)>n-1 \geq d+1.$ By Lemma \ref{lemma:link-bound-rigidity}, $\binom{\{x\} \cup C}{3}$ is a triangle component of $\T,$ where $C=C_x(y)$ and $|\{x\} \cup C|=d+2 \leq n.$ However, by Theorem \ref{thm:large-min-degree-overlap}, $d_e(\T) \geq \la(\T)-2$, and we conclude that $d+2=n$. Finally, since the two triangle components $\T'$ and $\binom{\{x\} \cup C}{3}$ support the edge $e$, they must be the same component. This completes the proof of the corollary. 
\end{proof}

\subsection{Rigidity and Extremal Constructions}
\label{subsec:rigidity-extremal-constructions}
Theorem \ref{thm:large-min-degree-overlap} gives local density around each vertex and edge in the support graph. The next counting lemma gives an upper bound on the number of vertices in $\Gamma_0(\T)$ in terms of $|\T|$ and $\lambda(\T),$ which will be crucial in the proofs of the rigidity results. 

\begin{lemma} \label{lemma:vertex-number-bound-by-lambda-t}
    Let $|\T|=t>0,\ \lambda := \lambda(\T) > 2,$ and $G=\Gamma_0(\T).$ If $v=|V(G)|$ and $e=|E(G)|,$ then:
    \begin{gather*}
        v \leq \frac{2e}{\lceil \lambda \rceil-1}, \quad e \leq \frac{3t}{\lceil \lambda \rceil-2}, \quad \text{and, } v \leq \frac{6t}{(\lceil \lambda \rceil-1)(\lceil \lambda \rceil-2)}.
    \end{gather*}
\end{lemma}

\begin{proof}
    Denote $n = \lceil \lambda \rceil$. By Theorem \ref{thm:large-min-degree-overlap}, each vertex $x$ is contained in at least $(n-1)$ edges in $G$. Thus, 
    \begin{gather*}
        v(n-1) \leq \sum_{x \in V(G)}d_G(x) = 2e \text{, or, } v \leq \frac{2e}{n-1}.
    \end{gather*}
    To prove the second inequality, we count the number of pairs $(p,T)$, where $p \in E(G), T \in \T,$ and $p$ is an edge of $T.$ By Theorem \ref{thm:large-min-degree-overlap}, each edge $p$ of $G$ is contained in at least $(n-2)$ triangles in $\T$, thus, the number of such pairs is at least $e(n-2).$ On the other hand, each triangle contains exactly three edges in $G$. Thus, the number of such pairs is exactly $3t.$ Thus,
    \begin{gather*}
        e(n-2) \leq 3t, \text{ or, } e \leq \frac{3t}{n-2}.
    \end{gather*}
    Combining these two inequalities, we obtain:
    \begin{gather*}
        v \leq \frac{2}{(n-1)(n-2)}(e(n-2)) \leq \frac{6t}{(n-1)(n-2)},
    \end{gather*}
    as needed.
\end{proof}

\begin{defn}
    For positive integer $t$, we define:
    \begin{gather*}
        \phi(t) = \max_{|\T|=t} \lambda(\T),
    \end{gather*}
    and,
    \begin{gather*}
        \Lambda(t) = \max_{1 \leq s \leq t} \phi(s).
    \end{gather*}
\end{defn}
Here, $\phi$ is well defined, since every family $\T$ of size $t$ uses at most $3t$ vertices, hence, up-to relabeling, there are only finitely many such families.

\begin{theorem} \label{thm:rigidity-lambda-at-C(n,3)}
    Let $n \geq 3$ be a positive integer. If $|\T| < \binom{n}{3}$, then $\lambda(\T)\leq n-1.$ Furthermore, if $|\T|=\binom{n}{3}$ and $\lambda(\T) > n-1,$ then $\lambda(\T)=n$, $\Gamma_0(\T)$ is isomorphic to $K_n$, and the triangles in $\T$ are precisely the set of all possible triangles on $V(\Gamma_0(\T)).$
\end{theorem}

\begin{proof}
    We prove the first statement first. Assume on the contrary that $\lambda(\T) > n-1.$ Then, $\lceil \lambda(\T) \rceil \geq n$. By Lemma \ref{lemma:vertex-number-bound-by-lambda-t},
    \begin{gather*}
        |V(\Gamma_0(\T))| \leq \frac{6|\T|}{(\lceil \lambda(\T)\rceil-1)(\lceil \lambda(\T)\rceil-2)} < \frac{6\binom{n}{3}}{(n-1)(n-2)}=n.
    \end{gather*}
    On the other hand, by Theorem \ref{thm:large-min-degree-overlap}, we have $|V(\Gamma_0(\T))| \geq \lceil \lambda(\T) \rceil \geq n,$ which is a contradiction. This proves the first statement.\par 

    Assume now that $|\T|=\binom{n}{3}$ and $\lambda(\T) > n-1$. Let $G=\Gamma_0(\T)$ and $v = |V(G)|.$ By Theorem \ref{thm:large-min-degree-overlap}, we have $v \geq \lceil \lambda(\T) \rceil \geq n.$ On the other hand, by Lemma \ref{lemma:vertex-number-bound-by-lambda-t}, we have
    \begin{gather*}
        v \leq \frac{6|\T|}{(\lceil \lambda(\T)\rceil-1)(\lceil \lambda(\T)\rceil-2)} \leq \frac{6\binom{n}{3}}{(n-1)(n-2)}=n.
    \end{gather*}
    Thus, $v = n.$ Since $|\T|=\binom{n}{3}$ and each triangle $T$ in $\T$ is a triangle in $G,$ we have that $\T$ is the set of all possible triangles in $G,$ and $G \cong K_n.$ Finally, by Example 2.3 in \cite{bgp}, we have $\lambda(\T)=n,$ as needed.
\end{proof}

\begin{corollary} \label{cor:phi-n}
    For any positive integer $n \geq 3,$ 
    \begin{gather*}
        \phi \left(\binom{n}{3}\right)=n,
    \end{gather*}
    and up to relabeling of the vertices, the unique $\T$ achieving this is the full set of triangles in $K_n.$ Equivalently,
    \begin{gather*}
        \min\{|\T| : \lambda(\T) \geq n\} = \binom{n}{3},
    \end{gather*}
    and the unique minimizer is the full triangle set of $K_n.$
\end{corollary}

\begin{proof}
    By standard computation of the spectrum of $\ul{1}(K_n)$ (see Example 2.3 in \cite{bgp}), the full set of triangles $\T$ of $K_n$ satisfies $\lambda(\T)=n.$ This shows that $\phi\left(\binom{n}{3}\right) \geq n.$ On the other hand, if $|\T|=\binom{n}{3}$ and $\lambda(\T)>n-1,$ then, by Theorem \ref{thm:rigidity-lambda-at-C(n,3)}, $\T$ must be the full triangle set of $K_n,$ and so, $\lambda(\T)=n.$ This proves that $\phi\left(\binom{n}{3}\right)=n.$ Furthermore, by Theorem \ref{thm:rigidity-lambda-at-C(n,3)}, we have $\phi(t) < n$ whenever $t < \binom{n}{3}.$ This proves the second statement.
\end{proof}

We have the following result as an immediate corollary.

\begin{theorem} \label{thm:exact-formula-Lambda}
    Let $t \geq 1$ be a positive integer. Then,
    \begin{gather*}
        \Lambda(t) = \max \left \{n \in \N : \binom{n}{3} \leq t\right\},
    \end{gather*}
    or, equivalently,
    \begin{gather*}
        \text{for every } n \in \N, \ \binom{n}{3} \leq t < \binom{n+1}{3} \implies \Lambda(t)=n.
    \end{gather*}
    Consequently,
    \begin{gather*}
        \Lambda(t) = \Theta(t^{\frac{1}{3}}).
    \end{gather*}
\end{theorem}

\begin{proof}
    It suffices to prove the first two statements. In Corollary \ref{cor:phi-n}, we have proved that $\min\{|\T| : \lambda(\T) \geq n\} = \binom{n}{3}$ and that $\phi\left(\binom{n}{3}\right)=n.$ This proves the first statement. Now, applying Theorem \ref{thm:rigidity-lambda-at-C(n,3)} with $n+1,$ we have that $\lambda(\T)\leq n$ whenever $|\T|< \binom{n+1}{3}.$ Therefore, if $t < \binom{n+1}{3},$ then $\Lambda(t) \leq n.$ On the other hand, $\phi\left(\binom{n}{3}\right)=n,$ and so, $\Lambda(t)\geq n$ whenever $t \geq \binom{n}{3}.$ Thus, if $\binom{n}{3} \leq t < \binom{n+1}{3},$ then $\Lambda(t)=n.$ 
\end{proof}

\begin{rk}
    Let $t \geq 1.$ If $n$ is the unique positive integer such that $\binom{n}{3}\leq t < \binom{n+1}{3},$ then by Theorem \ref{thm:exact-formula-Lambda} and a simple algebraic manipulation, we obtain:
    \begin{gather*}
        (6t)^{\frac{1}{3}}-1 \leq n \leq \Lambda(t) < n + 1< (6t)^{\frac{1}{3}}+3, \text{ or, } (6t)^{\frac{1}{3}}-1 \leq \Lambda(t) < (6t)^{\frac{1}{3}}+3.
    \end{gather*}
\end{rk}

\subsection{Non-monotonicity} \label{subsection:phi-nonmonotonicity}

In this subsection, we will show that for $n \geq 5,$ the $n$-threshold $\binom{n}{3}$ is isolated in the strongest possible sense: $\phi(t) \leq n-1$ on both the intervals $\left(1, \binom{n}{3}\right)$ and $\left( \binom{n}{3}, \binom{n+1}{3}\right).$ In particular, $\phi$ is non-monotone.

\begin{defn}
    Let $n \geq 3$ be a positive integer. Denote:
    \begin{gather*}
        H(n)=\min\left\{s \in \N : \phi\left(\binom{n}{3}+s\right) > n-1 \right\}.
    \end{gather*}
\end{defn}

Note that $\phi\left(\binom{n+1}{3}\right)=n+1,$ hence $H(n)$ is well-defined and $1 \leq H(n) \leq \binom{n}{2}.$

\begin{theorem}\label{thm:h(n)}
    Let $n \geq 5$. Then, $H(n)=\binom{n}{2}.$ Equivalently, $\phi(t) \leq n-1$ for every $\binom{n}{3} < t < \binom{n+1}{3}.$
\end{theorem}

\begin{proof}
    Since $\phi\left(\binom{n+1}{3}\right)=n+1,$ we have $H(n)\leq \binom{n}{2}.$ We will prove that $H(n) \geq \binom{n}{2}.$ Suppose, $\T$ satisfies $\binom{n}{3} < t=|\T| < \binom{n+1}{3}$ and $\la(\T) > n-1$. Let $\T'$ be a triangle component of $\T$. If $\T'$ satisfies the second case in Corollary \ref{cor:link-bound-}, then $d_{\min}(G') \geq n$, where $G'$ is the support graph of $\T'$. To see this, consider any edge $e=\{x,y\}$ supported on $\T'$. Then, there are $(n-1)$ triangles in $\T'$ containing the edge $e$. The $(n-1)$ (necessarily distinct) third vertices of these triangles contribute $(n-1)$ neighbors of $x$ and $y$ in $G'$. Since $x \sim_{G'} y$ and $e$ is an arbitrary edge of $G',$ this shows that $d_{\min}(G') \geq n.$ By counting the number of pairs $(e, T)$ with $e \in E(G'), T \in \T'$ and $e \subset T$ in two ways, we obtain that
    $$
    |\T| \geq \frac{1}{3}|E(G')|(n-1) \geq \frac{1}{6}|V(G')|n(n-1) \geq \binom{n+1}{3},
    $$
    a contradiction. Thus, by Corollary \ref{cor:link-bound-}, $\T'$ must be isomorphic to $\binom{[n]}{3}$, which holds for arbitrary triangle component $\T'$ of $\T$. However, if $\T$ has more than one triangle component, then
    $$
    t \geq 2\binom{n}{3} \geq \binom{n+1}{3}
    $$
    for $n \geq 5$, which is not possible. Thus, $\T$ contains only one triangle component, but in this case the hypothesis $\binom{n}{3} < t$ does not hold. Thus, this case is not possible either, and proof of the theorem is complete.
\end{proof}

\begin{corollary}
    For $n \geq 5,$ we have 
    \begin{gather*}
        \phi\left(\binom{n}{3}+1\right) \leq n-1, \quad \phi\left(\binom{n}{3}-1\right) < \phi\left(\binom{n}{3}\right) > \phi\left(\binom{n}{3}+1\right).
    \end{gather*}
\end{corollary}

\subsection{\texorpdfstring{The $2$-down Laplacian Spectrum of $K_c \vee \overline{K_b}$}{}}
\label{subsec:spectrum-glued-cliques}

In this section, we explicitly compute the spectrum of the graph $K_c \vee \overline{K_b}$ constructed by joining $b$ pairwise disjoint vertices to a copy of $K_c.$ Taking $\T$ to be the full triangle family of this graph, we thus obtain $\lambda(\T).$ This explicit value will be crucial in our discussions in the next section. In this section, we will identify the edges and triangles in a graph $G$ with the corresponding standard basis vector, when the ambient vector space is clear.

\begin{defn}
    Let $c \geq 3, b \geq 1$ be positive integers. Denote by $G_{c,b}$ the graph on vertex set $\{1,\dots,b+c\},$ and edge set:
    \begin{gather*}
        \big\{\{i,j\} : 1 \leq i < j \leq c \big\}\  \bigcup \ \big\{\{i,j\} : c+1 \leq i \leq c+b, \ 1 \leq j \leq c \big\}.
    \end{gather*}
\end{defn}

\begin{defn}
     Let $c \geq 3, b \geq 1$ be positive integers. Denote by $\T_{c,b}$ the set:
     \begin{gather*}
         \big\{ \{i,j,k\} : \{i,j,k\} \text{ is a } 3-\text{clique in } G_{c,b}\}.
     \end{gather*}
\end{defn}

Our goal in this subsection is to prove the following theorem.
\begin{theorem} \label{thm:spectra-of-G-c-b}
    Let $c \geq 3, b \geq 1$ be positive integers. Let $\dl{2,G_{c,b}}$ denote the clique $2$-down Laplacian of $G_{c,b}.$ Then, the spectrum of $\dl{2,G_{c,b}}$ is given by
    \begin{gather*}
        \begin{pmatrix}
            0 & c & b+c\\
            \binom{c}{3}+(b-1)\binom{c-1}{2} & (b-1)(c-1) & \binom{c}{2}
        \end{pmatrix}.
    \end{gather*}
    where eigenvalues with multiplicity $0$ are omitted.
\end{theorem}

Since $\Gamma_0(\T_{c,b})=G_{c,b},$ we have the following immediate corollary.

\begin{corollary} \label{cor:lambda-T-c-b}
    Let $c \geq 3, b \geq 1$ be positive integers. Then,
    \begin{gather*}
        \lambda(\T_{c,b}) = \begin{cases}
            c+1, & b=1\\
            c, & b \geq 2.
        \end{cases}.
    \end{gather*}
\end{corollary}
Note that $\T_{c,1}$ is the full set of triangles on a copy of $K_{c+1}$ and so, $\lambda(\T_{c,1})=c+1.$

We split the proof of Theorem \ref{thm:spectra-of-G-c-b} into several propositions.

\begin{prop} \label{prop:eigenvalue-verification-c}
     Let $c \geq 3, b \geq 2$ be positive integers. Let $\dl{2,G_{c,b}}$ denote the clique edge down-Laplacian of $G_{c,b}.$ Then, $c$ is an eigenvalue of $\dl{2,G_{c,b}}$ with multiplicity at least $(b-1)(c-1)$, and $(b-1)(c-1)$ linearly independent eigenvectors of $\dl{2,G_{c,b}}$ associated with the eigenvalue $c$ are given by:
     \begin{gather*}
        \sum_{1 \leq i \leq c, \ i \ne x} ([\{i,x,y\}:\{x,y\}] \cdot \{i,x,y\} - [\{i,x,b+c\}:\{x,b+c\}] \cdot \{i,x,b+c\}),
    \end{gather*}
    for $2 \leq x \leq c,\ c+1 \leq y \leq b+c-1.$
\end{prop}

\begin{proof}
    Denote $G=G_{c,b}$, $\T=\T_{c,b}$, $\delta_1=\delta_{1,G,\T},$ $\dl{2}=\dl{2,G,\T},$ and $\ul{1}=\ul{1,G,\T}.$ Fix $2 \leq x \leq c$ and $c+1 \leq y \leq b+c-1,$ and denote
    \begin{gather*}
        v_{x,y} := \sum_{1 \leq i \leq c, \ i \ne x} ([\{i,x,y\}:\{x,y\}] \cdot \{i,x,y\} - [\{i,x,b+c\}:\{x,b+c\}] \cdot \{i,x,b+c\}) \in \re^{\T}.
    \end{gather*}
    We prove that $v:=v_{x,y}$ is an eigenvector of $\dl{2}$ with eigenvalue $c.$ For any edge $e$ in $G$, we have:
    \begin{gather*}
        \left(\delta_1^Tv\right)_e=\begin{dcases}
            \sum_{\substack{1 \leq i \leq c,\\ i \ne x}} [\{i,x,y\}:\{x,y\}]^2=(c-1), & e = \{x,y\}\\
            \sum_{\substack{1 \leq i \leq c,\\ i \ne x}} -[\{i,x,b+c\}:\{x,b+c\}]^2=-(c-1), & e = \{x,b+c\}\\
            [\{i,x,y\}:\{x,y\}] \cdot [\{i,x,y\}:\{i,y\}]=-[\{i,x,y\}:\{i,x\}]=-1, & e = \{i,y\};\\
            & i \in [c], i \ne x\\
            -[\{i,x,b+c\}:\{x,b+c\}] \cdot [\{i,x,b+c\}:\{i,b+c\}] & e = \{i,b+c\};\\
            =[\{i,x,b+c\}:\{i,x\}] & i \in [c], i \ne x\\
            =1, & \\
            [\{i,x,y\}:\{x,y\}] \cdot [\{i,x,y\}:\{i,x\}] & e = \{i,x\};\\
             +  [\{i,x,b+c\}:\{i,b+c\}] \cdot [\{i,x,b+c\}:\{i,x\}] & i \in [c], i \ne x\\
            = -[\{i,x,y\}:\{i,y\}] + [\{i,x,b+c\}:\{i,b+c\}] 
            =1-1=0, \\
            0, & \text{otherwise.}
        \end{dcases}
    \end{gather*}
    Thus, for any $T \in \T,$ $\left(\dl{2}v\right)_T=\left( \delta_1\left( \delta_1^Tv\right)\right)_T$ is given by:
    \begin{gather*}
        \begin{dcases}
            0=c(v_{x,y})_T, & T \subset [c]\\
            0=c(v_{x,y})_T, & T=\{i,j,k\}; i,j \in [c],\\
            & c+1 \leq k \leq b+c-1, k \ne y\\
            (c-1)[T:\{x,y\}]-[T:\{i,y\}] & T=\{i,x,y\};i \in [c], i \ne x\\
            =[T:\{x,y\}]((c-1)+[T:\{i,x\}]) & \\
            =[T:\{x,y\}]((c-1)+1)=c(v_{x,y})_T & \\
            -(c-1)[T:\{x,b+c\}]+[T:\{i,b+c\}] & T=\{i,x,b+c\};i \in [c], i \ne x\\
            =[T:\{x,b+c\}](-c+1-[T:\{i,x\}]) & \\
            =[T:\{x,b+c\}](-c+1-1)=c(v_{x,y})_T, & \\
            -[T:\{i,y\}]-[T:\{j,y\}]=0=c(v_{x,y})_T & T=\{i,j,y\}; \{i,j\}\subset [c], i,j \neq x\\
            -[T:\{i,b+c\}]-[T:\{j,b+c\}]=0=c(v_{x,y})_T & T=\{i,j,b+c\}; \{i,j\}\subset [c], i,j \neq x.\\
        \end{dcases}
    \end{gather*}
    Since these cases cover all possible triangles in $G$, we conclude that $v_{x,y}$ is an eigenvector of $\dl{2}$ with eigenvalue $c.$ Furthermore, in $v_{x,y},$ the coordinate indexed by the triangle $\{x,y,1\}$ is nonzero, while this coordinate is $0$ for any $v_{x',y'}$ if $(x',y') \ne (x,y)$. Thus, the constructed $(b-1)(c-1)$ vectors $v_{x,y}$ are linearly independent. This shows that $c$ is an eigenvalue of $\dl{2}$ with multiplicity at least $(b-1)(c-1),$ and the proposition follows.
\end{proof}

\begin{prop} \label{prop:eigenvalue-verification-b+c}
     Let $c \geq 3, b \geq 1$ be positive integers. Let $\ul{1,G_{c,b}}$ denote the clique edge up-Laplacian of $G_{c,b}.$ Then, $b+c$ is an eigenvalue of $\ul{1,G_{c,b}}$, and $\binom{c}{2}$ linearly independent eigenvectors of $\ul{1, G_{c,b}}$ associated with the eigenvalue $b+c$ are given by:
     \begin{gather*}
        \{x,y\} + \frac{1}{b}\sum_{i=c+1}^{b+c} \left(-\{x,i\} + \{y,i\}\right)
    \end{gather*}
    for $x,y \in [c],\ x < y.$
\end{prop}
We provide a proof in Appendix \ref{appendix:eigenvalue-verification-b+c}.

Since the nonzero part of the spectra of $\ul{1}$ and $\dl{2}$ coincide, we have the following immediate corollary.

\begin{corollary}
    Let $c \geq 3, b \geq 1$ be positive integers. Let $\dl{2,G_{c,b}}$ denote the clique $2$-down Laplacian of $G_{c,b}.$ Then, $b+c$ is an eigenvalue of $\dl{2,G_{c,b}}$ with multiplicity at least $\binom{c}{2}.$
\end{corollary}

We can now prove Theorem \ref{thm:spectra-of-G-c-b}.

\begin{proof} [Proof of Theorem \ref{thm:spectra-of-G-c-b}]
    If $T \in \T$, we have:
    \begin{gather*}
        \left(\dl{2,G_{c,b}, \T_{c,b}}\right)_{T,T} = 3.
    \end{gather*}
    Thus, 
    \begin{gather*}
        \text{tr}\left(\dl{2,G_{c,b}, \T_{c,b}}\right) = 3|\T|=3\left(\binom{c}{3}+b\binom{c}{2}\right).
    \end{gather*}
    On the other hand,
    \begin{gather*}
        (b-1)(c-1)c+ \binom{c}{2} (b+c) = 3\left(\binom{c}{3}+b\binom{c}{2}\right) = \text{tr}\left(\dl{2,G_{c,b}, \T_{c,b}}\right).
    \end{gather*}
    Note that if $b=1$, then $(b-1)(c-1)c=0$, hence the omitted eigenvalue $c$ does not contribute to the sum in this case. Since $\dl{2,G_{c,b}, \T_{c,b}}$ is positive semidefinite, all remaining eigenvalues must be $0$. We obtain the multiplicity of the $0$ eigenvalue as:
    \begin{gather*}
        |\T|-(b-1)(c-1)-\binom{c}{2} = \binom{c}{3}+b\binom{c}{2}-(b-1)(c-1)- \binom{c}{2}=\binom{c}{3}+(b-1)\binom{c-1}{2}.
    \end{gather*}
    This proves the theorem.
\end{proof}

\subsection{\texorpdfstring{Growth rate of $\phi$}{}}
\label{subsec:growth-rate-phi}

In this subsection, we derive the explicit growth rate of $\phi.$ Despite the highly nonmonotone characteristic of $\phi$, we observe that $\phi$ has the same cubic root growth rate as $\Lambda.$ Denote by $\N_0$ the set of non-negative integers. We need the following number theoretic lemma.

\begin{lemma} \label{lemma:binomial-sum-lemma}
    Let $a \in \N, a \geq 3$. For every positive integer $N \geq 3a\binom{a}{2}+2a(a-1),$ there exist $x,y,z \in \N_0$ such that
    \begin{gather*}
        N = x \binom{a}{2} + y \binom{a+1}{2} + z \binom{a+2}{2}.
    \end{gather*}
    Consequently, for every positive integer $N \geq 2a^3+2a^2+1,$ there exist $x,y,z \in \N$ such that
    \begin{gather*}
        N = \left(\binom{a}{3} + x \binom{a}{2}\right) + \left(\binom{a+1}{3} + y \binom{a+1}{2}\right) + \left(\binom{a+2}{3} + z \binom{a+2}{2}\right).
    \end{gather*}
\end{lemma}
% $\binom{a}{3}+\binom{a+1}{3}+\binom{a+2}{3}+a\binom{a}{2}+a(2a-1)$

\begin{proof}
    For $q \in \N$ denote:
    \begin{gather*}
        I_q = \left\{n \in \N : q\binom{a}{2}+2a(a-1) \leq n \leq q\binom{a}{2} + qa\right\}. 
    \end{gather*}
     We claim that if $N \in I_q,$ then, there exist $x,y,z \in \N_0$ with $x+y+z=q$ and
     \begin{gather*}
         N=x\binom{a}{2}+y\binom{a+1}{2}+z\binom{a+2}{2}.
     \end{gather*}
     To prove this claim, first observe that 
     \begin{gather*}
         2a(a-1) \leq N-q\binom{a}{2} \leq qa.
     \end{gather*}
     However, $(a,2a+1)=1,$ and $a(2a+1)-a-(2a+1)+1=2a(a-1).$ By Frobenius, there exist $y,z \in \N_0$ such that
     \begin{gather*}
         N-q\binom{a}{2}=ya + z(2a+1)=y\left(\binom{a+1}{2}-\binom{a}{2} \right) + z\left( \binom{a+2}{2}-\binom{a}{2} \right).
     \end{gather*}
     Or,
     \begin{gather*}
         N=(q-y-z)\binom{a}{2} + y\binom{a+1}{2} + z \binom{a+2}{2}.
     \end{gather*}
     However, since $N-q\binom{a}{2} \leq qa,$ we have:
     \begin{gather*}
         (y+z)a \leq ya + z(2a+1) \leq qa, \text{ or, } y+z \leq q.
     \end{gather*}
     That is, $q-y-z \geq 0.$ Thus, taking $x=q-y-z \in \N_0,$ we have the desired expression of arbitrary $N \in I_q.$

     Now, if $q \geq 3a,$ then,
     \begin{gather*}
         q\binom{a}{2}+2a(a-1) < (q+1)\binom{a}{2}+2a(a-1) \leq q\binom{a}{2}+qa.
     \end{gather*}
     Therefore, for $q \geq 3a,$ we have $I_q \cap I_{q+1} \ne \emptyset.$ Consequently,
     \begin{gather*}
         \bigcup_{q=3a}^{\infty} I_q = \left\{n \in \N: n \geq 3a\binom{a}{2}+2a(a-1)\right\}.
     \end{gather*}
     That is, if $N \geq 3a\binom{a}{2}+2a(a-1),$ then there exist $x,y,z\in \N_0$ such that:
     \begin{gather*}
         N=x\binom{a}{2}+y\binom{a+1}{2}+z\binom{a+2}{2}.
     \end{gather*}
     Consequently, for every 
     \begin{gather*}
      N \geq 3a\binom{a}{2}+2a(a-1)+\sum_{i=0}^{2}\left(\binom{a+i}{3}+\binom{a+i}{2}\right)=2a^3+2a^2+1,
     \end{gather*}
    there exist $x,y,z\in \N$ such that:
     \begin{gather*}
        N = \left(\binom{a}{3} + x \binom{a}{2}\right) + \left(\binom{a+1}{3} + y \binom{a+1}{2}\right) + \left(\binom{a+2}{3} + z \binom{a+2}{2}\right).
    \end{gather*}
\end{proof}

\begin{theorem} \label{thm:phi-growth-rate}
    For every positive integer $a \geq 3,$
    \begin{gather*}
        t \in \N, t \geq 2a^3+2a^2+1 \implies \phi(t) \geq a.
    \end{gather*}
    Consequently, for $t \in \N, t \geq 81,$
    \begin{gather*}
        \phi(t) \geq \left \lfloor \left(\frac{t}{3}\right)^{\frac{1}{3}} \right \rfloor.
    \end{gather*}
    Furthermore, 
    \begin{gather*}
        \phi(t) = \Theta(t^{\frac{1}{3}}).
    \end{gather*}
\end{theorem}

\begin{proof}
    Let $a, N, x, y, z \in \N, a \geq 3$ be such that 
    \begin{gather*}
        N = \left(\binom{a}{3} + x \binom{a}{2}\right) + \left(\binom{a+1}{3} + y \binom{a+1}{2}\right) + \left(\binom{a+2}{3} + z \binom{a+2}{2}\right).
    \end{gather*}
    Then, take $\T$ to be the disjoint union of the full triangle families $\T_{a,x}, \T_{a+1,y},$ and $\T_{a+2,z}.$ Then,
    \begin{gather*}
        |\T|=\left(\binom{a}{3} + x \binom{a}{2}\right) + \left(\binom{a+1}{3} + y \binom{a+1}{2}\right) + \left(\binom{a+2}{3} + z \binom{a+2}{2}\right)=N.
    \end{gather*}
    Let $G=\Gamma_0(\T)$. Then, $\dl{2,G,\T}$ has a block diagonal form with three blocks, where the blocks correspond to the $2$-down Laplacian matrix of $\Gamma_0(\T_{a,x}),\Gamma_0(\T_{a+1,y}),$ and $\Gamma_0(\T_{a+2,z}).$ Thus,
    \begin{gather*}
        \lambda(\T)=\min\left\{ \lambda(\T_{a,x}), \lambda(\T_{a+1,y}), \lambda(\T_{a+2,z})\right\}\geq \min\{a,a+1,a+2\} \geq a.
    \end{gather*}
    Thus, $\phi(N) \geq a.$
    
    Now, by Lemma \ref{lemma:binomial-sum-lemma}, for fixed $a \in \N, a \geq 3$, for any $t \geq 2a^3+2a^2+1$, there exist $x,y,z \in \N$ with
     \begin{gather*}
        t=\left(\binom{a}{3} + x \binom{a}{2}\right) + \left(\binom{a+1}{3} + y \binom{a+1}{2}\right) + \left(\binom{a+2}{3} + z \binom{a+2}{2}\right).
    \end{gather*}
    By the observation above, $\phi(t) \geq a.$ That is, $\phi(t) \geq a$ whenever $t \geq 2a^3+2a^2+1.$ This proves the first statement. 

    Now, fix $t \in \N, t \geq 81.$ Set $a = \left \lfloor \left(\frac{t}{3}\right)^{\frac{1}{3}} \right \rfloor$. Then,
    \begin{gather*}
        2a^3+2a^2+1 < 3a^3 \leq 3\left(\left(\frac{t}{3}\right)^{\frac{1}{3}}\right)^3=t.
    \end{gather*}
    By the discussions above,
    \begin{gather*}
        \phi(t) \geq a = \left \lfloor \left(\frac{t}{3}\right)^{\frac{1}{3}} \right \rfloor.
    \end{gather*}
    Using the elementary fact that $x \in \re, x \geq 1 \implies \lfloor x \rfloor \geq \frac{x}{2},$ we have that
    \begin{gather*}
        \phi(t) \geq \frac{1}{2} \left(\frac{t}{3}\right)^{\frac{1}{3}}.
    \end{gather*}
    This proves that $\phi(t)=\Omega(t^{\frac{1}{3}}).$ Furthermore, by Theorem \ref{thm:exact-formula-Lambda}, $\Lambda(t)=\Theta(t^{\frac{1}{3}}),$ and $\phi(t) \leq \Lambda(t)$ for every $t$ by definition. Thus, $\phi(t)=O(t^{\frac{1}{3}}).$ We thus conclude that $\phi(t)=\Theta(t^{\frac{1}{3}}).$
\end{proof}

\section{Conclusion}
\label{sec:conclusion}
We studied the smallest positive eigenvalue $\lambda(\T)$ of the restricted edge up-Laplacian of a triangle set $\T$, treating it as a combinatorial parameter of the set itself. We showed that large $\lambda(\T)$ forces strong overlap among the triangles of $\T,$ and forces the support graph $\Gamma_0(\T)$ to have a large minimum degree. We also showed that $\la(\T) \leq 1 + \mu$ for every nonzero Laplacian eigenvalue $\mu$ of every connected component of every vertex link (Theorem \ref{thm:lambda-link-bound}), which gives the rigidity statement of Lemma \ref{lemma:link-bound-rigidity}. We also proved that $\binom{n}{3}$ is the exact threshold for attaining level $n$ with rigidity at equality, and that for $n \geq 5$ this threshold is isolated: $\phi(t) \leq n-1$ for $\binom{n}{3} < t < \binom{n+1}{3}$ (Theorem \ref{thm:h(n)}). Moreover, at the global scale, we determine the exact behavior of the monotone envelope $\Lambda:$
\begin{gather*}
    \Lambda(t) = \max\left\{ n \in \N : \binom{n}{3} \leq t \right\}.
\end{gather*}
Thus, the cubic-root growth law $\Lambda(t) =\Theta(t^{\frac{1}{3}})$ appears as an immediate corollary, while the same growth rate of $\phi(t)=\Theta(t^{\frac{1}{3}})$ required more sophisticated arguments.

Taken together, these results show that $\lambda(\T)$ already supports a meaningful structural and extremal picture: it detects local overlap, forces a support of large minimum degree, exhibits rigid clique thresholds and non-monotonicity, and has a monotone envelope $\Lambda$ with an exact staircase description.

Two natural problems remain open. First, the exact value of $\phi(t)$ inside each interval $\binom{n}{3} < t < \binom{n+1}{3}$ is not known; Theorem \ref{thm:h(n)} only gives the upper bound $n-1.$ Second, one would like a structural description of extremal and near-extremal families: when $\lambda(\T)$ is large relative to $|\T|$, or close to the staircase level $\Lambda(|\T|)$, how close must the support graph $\Gamma_0(\T)$ be to a clique or other rigid dense configuration? These questions indicate that the structural and extremal study of $\lambda$ and $\phi$ is only beginning. 

\section*{Declaration of generative AI and AI-assisted technologies}

During the preparation of this work, the author used OpenAI's ChatGPT and Anthropic's Claude for auxiliary support in identifying potentially relevant literature, checking computational approaches used to verify examples, and language polishing. The author remains fully responsible for the article.

\printbibliography

\newpage 

\appendix

\section{Proof of Proposition \ref{prop:eigenvalue-verification-b+c}} 
\label{appendix:eigenvalue-verification-b+c}

\begin{proof}
    Denote $G=G_{c,b}$, $\T=\T_{c,b}$, $\delta_1=\delta_{1,G,\T},$ $\dl{2}=\dl{2,G,\T},$ and $\ul{1}=\ul{1,G,\T}.$ For $1 \leq x < y \leq c,$ denote
    \begin{gather*}
        w_{x,y} := \{x,y\} + \frac{1}{b}\sum_{i=c+1}^{b+c} \left(-\{x,i\} + \{y,i\}\right) \in \re^{E(G)}.
    \end{gather*}
    We prove that $w_{x,y}$ is an eigenvector $\ul{1}$ with eigenvalue $b+c.$ Fix $1 \leq x < y \leq c$ and denote $w=w_{x,y}$. For any triangle $T \in \T,$ we have:
    \begin{gather*}
        \left( \delta_1w \right)_T=\begin{dcases}
            0, & \  \text{for every edge $e$ of $T$, }w_e=0\\
            % [T:\{x,y\}], & \{x,y\} \subset T \subset [c]\\
            [T:\{x,y\}]+\frac{1}{b}[T:\{i,y\}]-\frac{1}{b}[T:\{i,x\}]& T=\{i,x,y\},\ i \geq c+1\\
            =1 + \frac{2}{b}, & \\
            -\frac{1}{b}[T:\{i,x\}], & T=\{i,x,z\},i \geq c+1, \ z \in [c],\ z \ne y\\
            \frac{1}{b}[T:\{i,y\}], & T=\{i,y,z\},i \geq c+1,\ z \in [c],\ z \ne x\\
            [T:\{x,y\}]=1, & T=\{k,x,y\},\ k \in [c] \setminus \{x,y\},\ k < x \text{ or } y < k\\
            [T:\{x,y\}]=-1, & T=\{k,x,y\},\ k \in [c] \setminus \{x,y\},\ x < k < y.
        \end{dcases}
    \end{gather*}
    Thus, for any edge $e$ in $G$, $\left(\ul{1}w\right)_e = \left(\delta_1^T \left( \delta_1w \right)\right)_e$ is given by:
    \begin{gather*}
        \begin{dcases}
            \sum_{i=c+1}^{b+c}\left(1+\frac{2}{b}\right)[\{i,x,y\}:\{x,y\}] + \sum_{\substack{k=1,\\ k \ne x, y}}^c[\{k,x,y\}:\{x,y\}]^2& e=\{x,y\}\\
            =\left(1+\frac{2}{b}\right)b+1 \cdot (c-2)=b+c=(b+c)w_e, &\\
            [\{k,x,y\}:\{x,y\}][\{k,x,y\}:\{k,x\}]+ \sum_{i=c+1}^{b+c}\left(-\frac{1}{b}\right)[\{i,k,x\}:\{k,x\}] & e = \{k,x\},\ k \in [c] \setminus \{x,y\},\ k< x\\
            =1 - b\frac{1}{b}=0=(b+c)w_e, & \\
            [\{k,x,y\}:\{x,y\}][\{k,x,y\}:\{k,x\}]+ \sum_{i=c+1}^{b+c}\frac{1}{b} & e = \{k,x\},\ k \in [c] \setminus \{x,y\},\ y < k\\
            =- 1 + b\frac{1}{b}=0=(b+c)w_e, & \\
            [\{k,x,y\}:\{x,y\}][\{k,x,y\}:\{k,x\}]+ \sum_{i=c+1}^{b+c}\frac{1}{b}[\{i,k,x\}:\{k,x\}] & e = \{k,x\},\ k \in [c] \setminus \{x,y\},\ x < k <y\\
            =-1 + b\left(\frac{1}{b}\right) =0=(b+c)w_e, & \\
            % \color{cyan}
            % [\{k,x,y\}:\{x,y\}][\{k,x,y\}:\{k,y\}]+ \sum_{i=c+1}^{b+c}\frac{1}{b}[\{i,k,y\}:\{k,y\}] & e = \{k,y\},\ k \in [c] \setminus \{x,y\},\ k < x\\
            % =-1 + b\frac{1}{b}=0=(b+c)w_e, & \\
            % \color{red}
            % [\{k,x,y\}:\{x,y\}][\{k,x,y\}:\{k,y\}]+ \sum_{i=c+1}^{b+c}\frac{1}{b}[\{i,k,y\}:\{k,y\}] & e = \{k,y\},\ k \in [c] \setminus \{x,y\},\ x < k < y\\
            % =-1 + b\frac{1}{b}=0=(b+c)w_e, & \\
            % % \color{blue}
            % [\{k,x,y\}:\{x,y\}][\{k,x,y\}:\{k,y\}]+ \sum_{i=c+1}^{b+c}\left(-\frac{1}{b}\right)[\{i,k,y\}:\{k,y\}] & e = \{k,y\},\ k \in [c] \setminus \{x,y\},\ y < k\\
            % =1 + b\left(-\frac{1}{b}\right)=0=(b+c)w_e, & \\
            % \left(\ul{1}\right)_{e,e}\left(-\frac{1}{b}\right)+\left(\ul{1}\right)_{e,\{x,y\}}1+\left(\ul{1}\right)_{e,\{i,y\}}\left(\frac{1}{b}\right) & e = \{i,x\},\ c+1 \leq i \leq b+c\\
            % = -\frac{c-1}{b}-1-\frac{1}{b}=-\frac{b+c}{b}=(b+c)w_{e}, & \\
            % \left(\ul{1}\right)_{e,e}\left(\frac{1}{b}\right)+\left(\ul{1}\right)_{e,\{x,y\}}1+\left(\ul{1}\right)_{e,\{i,x\}}\left(-\frac{1}{b}\right) & e = \{i,y\},\ c+1 \leq i \leq b+c\\
            % = \frac{c-1}{b}+1+\frac{1}{b}=\frac{b+c}{b}=(b+c)w_{e}, & \\
            % \left(\ul{1}\right)_{e,\{i,x\}}\left(-\frac{1}{b}\right)+\left(\ul{1}\right)_{e,\{i,y\}}\left( \frac{1}{b}\right) = \frac{1}{b}-\frac{1}{b} & e = \{i,z\},\ c+1 \leq i \leq b+c, \\
            % 0 = (b+c)w_e & z \in [c],\ z \ne x,y\\
            % 0, & e=\{z,w\},\ z,w \in [c] \setminus \{x,y\}.
        \end{dcases}
    \end{gather*}
    (continued)
    \begin{gather*}
        \begin{dcases}
            % \sum_{i=c+1}^{b+c}\left(1+\frac{2}{b}\right)[\{i,x,y\}:\{x,y\}] + \sum_{\substack{k=1,\\ k \ne x, y}}^c[\{k,x,y\}:\{x,y\}]^2& e=\{x,y\}\\
            % =\left(1+\frac{2}{b}\right)b+1 \cdot (c-2)=b+c=(b+c)w_e, &\\
            % [\{k,x,y\}:\{x,y\}][\{k,x,y\}:\{k,x\}]+ \sum_{i=c+1}^{b+c}\left(-\frac{1}{b}\right)[\{i,k,x\}:\{k,x\}] & e = \{k,x\},\ k \in [c] \setminus \{x,y\},\ k< x\\
            % =1 - b\frac{1}{b}=0=(b+c)w_e, & \\
            % [\{k,x,y\}:\{x,y\}][\{k,x,y\}:\{k,x\}]+ \sum_{i=c+1}^{b+c}\frac{1}{b} & e = \{k,x\},\ k \in [c] \setminus \{x,y\},\ y < k\\
            % =- 1 + b\frac{1}{b}=0=(b+c)w_e, & \\
            % [\{k,x,y\}:\{x,y\}][\{k,x,y\}:\{k,x\}]+ \sum_{i=c+1}^{b+c}\frac{1}{b}[\{i,k,x\}:\{k,x\}] & e = \{k,x\},\ k \in [c] \setminus \{x,y\},\ x < k <y\\
            % =-1 + b\left(\frac{1}{b}\right) =0=(b+c)w_e, & \\
            % % \color{cyan}
            [\{k,x,y\}:\{x,y\}][\{k,x,y\}:\{k,y\}]+ \sum_{i=c+1}^{b+c}\frac{1}{b}[\{i,k,y\}:\{k,y\}] & e = \{k,y\},\ k \in [c] \setminus \{x,y\},\ k < x\\
            =-1 + b\frac{1}{b}=0=(b+c)w_e, & \\
            % \color{red}
            [\{k,x,y\}:\{x,y\}][\{k,x,y\}:\{k,y\}]+ \sum_{i=c+1}^{b+c}\frac{1}{b}[\{i,k,y\}:\{k,y\}] & e = \{k,y\},\ k \in [c] \setminus \{x,y\},\ x < k < y\\
            =-1 + b\frac{1}{b}=0=(b+c)w_e, & \\
            % \color{blue}
            [\{k,x,y\}:\{x,y\}][\{k,x,y\}:\{k,y\}]+ \sum_{i=c+1}^{b+c}\left(-\frac{1}{b}\right)[\{i,k,y\}:\{k,y\}] & e = \{k,y\},\ k \in [c] \setminus \{x,y\},\ y < k\\
            =1 + b\left(-\frac{1}{b}\right)=0=(b+c)w_e, & \\
            \left(\ul{1}\right)_{e,e}\left(-\frac{1}{b}\right)+\left(\ul{1}\right)_{e,\{x,y\}}1+\left(\ul{1}\right)_{e,\{i,y\}}\left(\frac{1}{b}\right) & e = \{i,x\},\ c+1 \leq i \leq b+c\\
            = -\frac{c-1}{b}-1-\frac{1}{b}=-\frac{b+c}{b}=(b+c)w_{e}, & \\
            \left(\ul{1}\right)_{e,e}\left(\frac{1}{b}\right)+\left(\ul{1}\right)_{e,\{x,y\}}1+\left(\ul{1}\right)_{e,\{i,x\}}\left(-\frac{1}{b}\right) & e = \{i,y\},\ c+1 \leq i \leq b+c\\
            = \frac{c-1}{b}+1+\frac{1}{b}=\frac{b+c}{b}=(b+c)w_{e}, & \\
            \left(\ul{1}\right)_{e,\{i,x\}}\left(-\frac{1}{b}\right)+\left(\ul{1}\right)_{e,\{i,y\}}\left( \frac{1}{b}\right) = \frac{1}{b}-\frac{1}{b} & e = \{i,z\},\ c+1 \leq i \leq b+c, \\
            0 = (b+c)w_e & z \in [c],\ z \ne x,y\\
            0, & e=\{z,w\},\ z,w \in [c] \setminus \{x,y\}.
        \end{dcases}
    \end{gather*}
    Thus, $w=w_{x,y}$ is an eigenvector of $\ul{1}$ with eigenvalue $(b+c).$ Furthermore, the coordinate indexed by the edge $\{x,y\}$ is nonzero in $w_{x,y}$, while the same coordinate is $0$ in $w_{x',y'}$ for any $\{x',y'\} \ne \{x,y\}.$ Thus, the multiplicity of the eigenvalue $(b+c)$ of $\ul{1}$ is at least $\binom{c}{2}.$ This proves the proposition.
\end{proof}

\end{document}